\documentclass[english, 11pt]{amsart}
\usepackage{amssymb,amsbsy,amsmath,amsfonts,amssymb,amscd}
\usepackage{aliascnt}
\usepackage{mathrsfs,graphics,color,graphicx}
\usepackage[all,poly,knot]{xy}
\usepackage{hyperref,rotating,ifthen,wasysym}
\usepackage{comment,textcomp}  
\usepackage{csquotes,footmisc,fancyhdr,fancybox,}   
\usepackage{arydshln}
\usepackage{times}
\usepackage{pdfpages}
\usepackage{multicol,multirow,longtable}
\usepackage{nccrules}
\usepackage{array}
\usepackage{framed}
\usepackage{pgf,tikz}
\usetikzlibrary{arrows}
\usepackage{lipsum}
\usepackage{mathtools}
\usepackage[top=1.5in, bottom=1.5in, left=1in, right=1in]{geometry}
\usepackage[hyperpageref]{backref}

\usepackage{xcolor}
\hypersetup{
	colorlinks,
	linkcolor={red!50!black},
	citecolor={blue!62!black},
	urlcolor={blue!80!black}
}

\def\charac{\mathop{\rm char}\nolimits}

\def\Mult{\mathop{\rm Mult}\nolimits}

\newcommand{\K}{{\mathbb{K}}}

\theoremstyle{plain}
\newtheorem{theorem}{Theorem}[section]
\newtheorem{lemma}[theorem]{Lemma}
\newtheorem{corollary}[theorem]{Corollary}
\newtheorem{proposition}[theorem]{Proposition}
\theoremstyle{definition}

\numberwithin{equation}{section}

\begin{document}
	
	\title{Some line and conic arrangements and their Waldschmidt constants}
	
	\author{Dinh Tuan Huynh}
	\address{Department of Mathematics, University of Education - Hue University,
		34 Le Loi, Hue, Vietnam}
	\email{dinhtuanhuynh@hueuni.edu.vn}
	
	\author{Tran N.~K.~Linh}
	\address{Department of Mathematics, University of Education - Hue University,
		34 Le Loi, Hue, Vietnam}
	\email{tnkhanhlinh@hueuni.edu.vn}
	
	\author{Le Ngoc Long}
	\address{Department of Mathematics, University of Education - Hue University,
		34 Le Loi, Hue, Vietnam}
	\email{lelong@hueuni.edu.vn}

	
	\date{\today}
	
	\dedicatory{J\"urgen Herzog (1941-2024), in memoriam.}
	
	\begin{abstract}
	We study the Waldschmidt constant of some configurations in the projective plane. In the first part, we show that the Waldschmidt constant of a set $\mathbb{X}$ of $n$ points where at least  $n-3$ points among them lie on a line  is either equal to $1, \frac{2n-3}{n-1}, 2, \frac{16}{7}, \frac{7}{3}, \frac{17}{7},$ or $\frac{5}{2}$. Together with the Hilbert polynomials, this gives a complete geometric characterization for $\mathbb{X}$. Next, we study 
	some specific configurations whose Waldschmidt constants are bounded from above by $\frac{5}{2}$. Under this condition, we describe all configurations of $n$ points with $n-1$ points among them lying on an irreducible conic, and we also study some specific configurations of $9$ points.
	\end{abstract}
	
	\maketitle

%
%
\section{Introduction}
\label{introduction}
	
	In this paper, we work on an algebraically closed field $\K$ 
	of characteristic zero. Let $\mathbb{X}=\{P_1,\dots,P_n\}$ be 
	a finite set of $n$ points in the projective plane $\mathbb{P}^2$ 
	over $\K$. Let $I_{\mathbb{X}}$ be the saturated homogeneous ideal in 
	the homogeneous coordinate ring of $\mathbb{P}^2$ defining $\mathbb{X}$, then
	\[
	I_{\mathbb{X}} \; =\; I(P_1)\cap\cdots\cap I(P_n),
	\]
	where $I(P_i)$ denotes the associated homogeneous prime ideal 
	of the point $P_i$. For a positive integer $m$, the $m$-th {\it symbolic power} of $I_{\mathbb{X}}$ is defined as
	\[
	I_{\mathbb{X}}^{(m)} \; :=\;  I(P_1)^{(m)}\cap \cdots\cap I(P_n)^{(m)},
	\]
	whose elements are homogeneous polynomials vanishing to order 
	at least $m$ on $\mathbb{X}$. The study of symbolic ideals 
	of the set of points have been active in recent decades, 
	partly because the link with various problems in algebraic 
	geometry and commutative algebra, for instance the long 
	standing Nagata conjecture.
	
	It is well-known that the initial degree of the ideal 
	$I_{\mathbb{X}}^{(m)}$ can provide important information 
	about its structure. Here, recalling that for an arbitrary 
	homogeneous ideal $I=\oplus_{i\geq 0}I_i$, 
	the {\it initial degree} of $I$ is defined to be the number
	\[
	\alpha(I) \; :=\; \min \{\, i\in\mathbb{Z} \mid I_i\not=0 \,\}.
	\]
	
	Now let $m\mathbb{X}$ denote the subscheme of $\mathbb{P}^2$ 
	defined by the ideal $I_{\mathbb{X}}^{(m)}$. 
	We often write $\alpha_{m\mathbb{X}}$ instead of 
	$\alpha (I_{\mathbb{X}}^{(m)})$, since we are mainly interested 
	in the geometry of the set $\mathbb{X}$. 
	Interests for studying this constant come from various domains 
	of researches. For instance, when $\K=\mathbb{C}$, this number 
	appeared in a generalized version of the Schwarz Lemma in 
	several complex variables~\cite{Moreau1980}. 
	To obtain an estimate independent of $m$, one considers the 
	following related asymptotic number
	\[
	\hat{\alpha}_{\mathbb{X}} \; :=\; \lim_{m\rightarrow\infty}
	\dfrac{\alpha_{m\mathbb{X}}}{m},
	\]
	called the {\it Waldschmidt constant} for the set 
	$\mathbb{X}$~\cite{Waldschmidt1977}. 
	The existence of this limit was proved by Chudnovsky~\cite{Chudnovsky1979}. Furthermore, one actually has that \cite{Bocci-Harbourne2010}
	\[
	\hat{\alpha}_{\mathbb{X}} \;=\; \inf_{m\geq 1}
	\dfrac{\alpha_{m\mathbb{X}}}{m}.
	\]
	
	Many works have been done to obtain effective lower bounds 
	for $\alpha_{m\mathbb{X}}$. Researches in this direction 
	were motivated by the conjectures of Chudnovsky \cite{Chudnovsky1979} 
	and Demailly \cite{Demailly1982}, see \cite{Malara2018} 
	and references therein for more information. 
	
	On the other hand, although $\hat{\alpha}_{\mathbb{X}}$ is an 
	important asymptotic invariant for $\mathbb{X}$, computing 
	explicitly this constant is not an easy task in general. 
	Until now, this question has been settled only in some special 
	configurations. Notably, in \cite{Dumnicki-Szemberg-Tutaj2016}, 
	the authors characterized all configurations of points
	$\mathbb{X}\subset\mathbb{P}^2$ such that
	$\hat{\alpha}_{\mathbb{X}}<\frac{9}{4}$. 
	A detail description of the configuration $\mathbb{X}$ with
	$\hat{\alpha}_{\mathbb{X}}=2$ was given in~\cite{Mosakhani-Haghighi2016}. In \cite{Catalisano-Favacchio-Guardo-Shin2024}, \cite{Virginia-Elena-Su2020}, the Waldschmidt constants of some specific $\Bbbk$-configurations were computed. The study of the Waldschmidt constant of zero dimensional schemes and related topics is a very active research direction, see \cite{Bocci2016, Enrico2020, Michael2019, Dumnicki2015, Dumnicki2014}, just to cite a few,  for some recent progresses.
	
	In the current paper, we first consider the configuration
	$\mathbb{X}=\{P_1,\dots,P_n\}$ of $n\ge 7$ distinct points 
	in $\mathbb{P}^2$, in which the first $n-3$ points lie on
	a line $L$ and the remaining three points are out of $L$. 
	We compute explicitly the Waldschmidt constant for all 
	circumstances of the position of points in $\mathbb{X}$. 
	Together with the Hilbert polynomials, we provide a complete 
	geometric characterization for the being considered configuration. Next, we describe all configurations $\mathbb{X}$ of $n$ points with $n-1$ points among them lying on an irreducible conic such that  $\hat{\alpha}_{\mathbb{X}}\leq\frac{5}{2}$. Finally, in the last part, we consider some specific configurations of $9$ points whose Waldschmidt constant is exactly $\frac{5}{2}$. Due to the combinatorial complexity, we restrict the discussion to the case where the points lie on the union of an irreducible conic and a line.

	Our main tool in computing the Waldschmidt constant is the following notions introduced 
	in~\cite{Dumnicki-Szemberg-Tutaj2016}, related to the classical 
	B\'{e}zout Theorem, which shall play crucial role in estimating 
	the  intersection multiplicities. Let $D$ be an effective divisor 
	of degree $d$ on $\mathbb{P}^2$ vanishing to order $m_i$ 
	at the point $P_i$ ($1\leq i\leq n)$, and let
	$\mathcal{C}_1,\dots,\mathcal{C}_r$ be irreducible algebraic curves 
	of degree $d_1,\dots,d_r$, respectively, such that $\Mult_{P_i}\mathcal{C}_j=m_{ij}$ for $1\leq i\leq n,\,1\leq j\leq r$. Throughout this paper, with some abuse of notation, we use the same symbol for an algebraic curve and its associated divisor. By~\cite[Subsection 2.1]{Dumnicki-Szemberg-Tutaj2016}, we have
	\begin{proposition}
	\label{bezout decomposition}
	There exist uniquely integers $a_1,\dots, a_r\geq 0$ and the divisor $B(D)$ such that
	\[
	D \;=\; \sum_{j=1}^{r}a_j\mathcal{C}_j+B(D)
	\]
	and
	\begin{equation}\label{CountingMultiplicities}
		B(D)\cdot \mathcal{C}_j \;\geq \;
		\sum_{i=1}^{n}
		\bigg(
		m_i-\sum_{\ell=1}^{r}a_{\ell}m_{i\ell}
		\bigg)
		m_{ij},
		\qquad
		(1\,\leq\,j\,\leq\,r).
	\end{equation}

The above decomposition of $D$ is called its {\it B\'{e}zout decomposition} with respect to $\mathcal{C}_1,\dots, \mathcal{C}_r$,
and the term $B(D)$ is called the {\it B\'{e}zout reduction} of $D$.		
	\end{proposition}

%
%
\section{Waldschmidt constant for set of $n$ points with $n-3$ points on a line}
\label{sec-2:Waldschmidt}

In this section, we consider the Waldschmidt constant for a set 
$\mathbb{X} = \{P_1,\dots, P_n\}$ of $n\ge 7$ distinct points 
in $\mathbb{P}^2$ such that at least $n-3$ points among them are collinear. In the case where
there are at least $n-2$ points of $X$ lying on a line, it is well-known (see e.g. \cite{Dumnicki-Szemberg-Tutaj2016}) that:
\begin{enumerate}
	\item $\hat{\alpha}_{\mathbb{X}} =1$ if $\mathbb{X}$
	lies on a line;
	\item  $\hat{\alpha}_{\mathbb{X}} = \frac{2n-3}{n-1}$ 
	if exactly $n-1$ points of $\mathbb{X}$ lie on a line;
	\item $\hat{\alpha}_{\mathbb{X}} = 2$ 
	if exactly $n-2$ points of $\mathbb{X}$ lie on a line.
\end{enumerate}
Hence, it is enough to look at the case where exactly $n-3$ points are collinear. From now on, for two distinct points $P,Q\in\mathbb{P}^2$, we
denote by $PQ$ the line passing through them. In the following, we may assume without loss of generality that 
the first $n-3$ points $P_1, \dots, P_{n-3}$ of $\mathbb{X}$ lie on 
a line $L$, and we write $Q_1, Q_2, Q_3$ 
instead of $P_{n-2}, P_{n-1}, P_{n}$ for the remaining three points lying
out of $L$. In the simplest case where all $Q_i$ are collinear, the Waldschmidt constant of $\mathbb{X}$ was computed in \cite[Theorem A]{Mosakhani-Haghighi2016} as
follows.
\begin{lemma}
If the three points $Q_1,Q_2,Q_3$ are collinear, 
then $\hat{\alpha}_{\mathbb{X}}=2$.
\end{lemma}

Now let us consider the case where $Q_1, Q_2, Q_3$ are not collinear.
Set $L_1 := Q_2Q_3$, $L_2 := Q_1Q_3$, $L_3 := Q_1Q_2$
and $\mathcal{Q} := L_1\cup L_2\cup L_3$. 
A point on $\mathcal{Q}\setminus\{Q_1,Q_2,Q_3\}$ 
will  be called a \textit{$\mathcal{Q}$-collinear point}.
First we treat the case where $n=7$. 

\begin{proposition}\label{Prop-WaldschmidtConst-7Points}
Let $\mathbb{X} = \{P_1,P_2, P_3,P_4,Q_1,Q_2,Q_3\}\subseteq \mathbb{P}^2$
be a set of 7 points. 
\begin{enumerate}
	\item[(a)] If three of the points $P_1, P_2, P_3, P_4$ are $\mathcal{Q}$-collinear, then $\hat{\alpha}_{\mathbb{X}}=\frac{16}{7}.$
	\item[(b)] If exactly two of the points $P_1, P_2, P_3, P_4$ are $\mathcal{Q}$-collinear, then 
	$\hat{\alpha}_{\mathbb{X}}=\frac{7}{3}.$
	\item[(c)] If only one of the points $P_1, P_2, P_3, P_4$ is $\mathcal{Q}$-collinear, then 
	$\hat{\alpha}_{\mathbb{X}}=\frac{17}{7}.$
	\item[(d)] If none of the points $P_1, P_2, P_3, P_4$ is $\mathcal{Q}$-collinear, then 
	$\hat{\alpha}_{\mathbb{X}}=\frac{5}{2}.$
\end{enumerate} 
\end{proposition}
\begin{proof}
	
Proof for parts (a) and (b) of this result 
	can be found in \cite{FGHLMS2017}. Here we give a different proof 
	using the methods in \cite{Dumnicki-Szemberg-Tutaj2016}.
	
(a)\quad We may assume that $P_1, P_2, P_3$ lie on $L_1, L_2, L_3$, respectively
(see Figure~\ref{Figure001}). Recall that $L$ is the line passing through 
$P_1, P_2, P_3$, and $P_4$. Put $H_1 := P_4Q_1$, $H_2 :=P_4Q_2$ and 
$H_3 :=P_4Q_3$. Then the divisor 
$D' = m(H_1+H_2+H_3)+3m(L_1+L_3+L_3)+4mL$ is 
of degree $16m$ and vanishes to order $7m$ at every point of $\mathbb{X}$. 
Hence we get an upper bound $\hat{\alpha}_{\mathbb{X}}\le \frac{16}{7}$ for the Waldschmidt constant of $\mathbb{X}$.
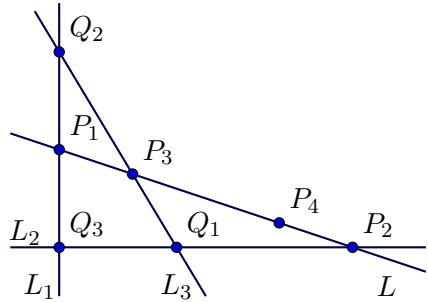
\begin{figure}[ht] 
	\begin{center}
		\begin{tikzpicture}[scale=0.65]
			\draw[-,blue!30!black,thick] (-1,0) -- (7.5,0);
			\draw[-,blue!30!black,thick] (0,-1) -- (0,5);
			\draw[-,blue!30!black,thick] (-0.5,29/6) -- (3,-1);
			\draw[-,blue!30!black,thick] (-1,7/3) -- (7.5,-0.5);
			\filldraw[color=black, fill=blue!80!black] (0,0) circle (3pt)
			node[anchor=south west]{$Q_3$};
			\filldraw[color=black, fill=blue!80!black] (2.4,0) circle (3pt)
			node[anchor=south west]{$Q_1$};
			\filldraw[color=black, fill=blue!80!black] (0,4) circle (3pt)
			node[anchor=south west]{$Q_2$};
			\filldraw[color=black, fill=blue!80!black] (6,0) circle (3pt)
			node[anchor=south west]{$P_2$};
			\filldraw[color=black, fill=blue!80!black] (0,2) circle (3pt)
			node[anchor=south west]{$P_1$};
			\filldraw[color=black, fill=blue!80!black] (1.5,1.5) circle (3pt)
			node[anchor=south west]{$P_3$};
			\filldraw[color=black, fill=blue!80!black] (4.5,0.5) circle (3pt)
			node[anchor=south west]{$P_4$};
			\node (L) at (6.7,-0.8) {$L$};
			\node (L1) at (-0.4,-0.8) {$L_1$};
			\node (L2) at (-0.7,0.3) {$L_2$};
			\node (L3) at (2.4,-0.8) {$L_3$};
		\end{tikzpicture}\vspace*{-0.3cm}
	\end{center}
	\caption{Seven points with three $\mathcal{Q}$-collinear points \label{Figure001}}
\end{figure}

Next, to show  $\hat{\alpha}_{\mathbb{X}}\ge \frac{16}{7}$, 
let $D$ be a divisor of degree $d$ vanishing along $m\mathbb{X}$.
The B\'ezout decomposition of $D$ with respect to the set $\mathbb{X}$ 
and lines $L_1, L_2, L_3, L, H_1, H_2$ and $H_3$ is 
$$
D = p(L_1+L_2+L_3)+qL+r(H_1+H_2+H_3)+B(D),
$$
where $B(D)$ is the B\'{e}zout reduction and $p,q,r\in\mathbb{N}$. 
Applying Proposition~\ref{bezout decomposition} for the curves $B(D)$ and $H_1,L_1,L$, we have
$$
\begin{aligned}
	B(D)\cdot H_1&=d-3p-q-3r \ge (m-q-3r)+(m-2p-r),\\
	B(D)\cdot L_1&=d-3p-q-3r \ge 2(m-2p-r)+(m-p-q),\\
	B(D)\cdot L&=d-3p-q-3r \ge 3(m-p-q)+(m-q-3r).
\end{aligned}
$$
Combining with the condition $B(D)\ge 0 $ gives a system of inequations
\begin{align}
	d&\ge 3p+q+3r,  \quad \label{eq1}\\
	d&\ge 2m+p-r,  \quad  \label{eq2}\\
	d&\ge 3m+r-2p,  \quad  \label{eq3}\\
	d&\ge 4m-3q.  \quad  \label{eq4} 
\end{align}
Multiplying the inequations (\ref{eq1}), (\ref{eq2}), (\ref{eq3}) 
and (\ref{eq4})  by $3$; $27$; $18$ and $1$, respectively, 
and then adding side-by-side the resulting inequations, we get
$d\ge \frac{16}{7}m.$ It  follows that $\hat{\alpha}_{\mathbb{X}}\ge \frac{16}{7}$, as wanted.

\medskip\noindent (b)\quad 
Suppose that $P_2, P_3$ lie on $L_2, L_3$, respectively
(see Figure~\ref{Figure002}). Let $C$ be the irreducible conic 
passing through $Q_1, Q_2, Q_3, P_1$, $P_4$. Then the divisor $D'=m(L_1+L_2+L_3+C+2L)$ is of degree $7m$ 
and vanishes to order $3m$ at every point of $\mathbb{X}$. 
This implies $\hat{\alpha}_{\mathbb{X}}\le \frac{7}{3}.$
\begin{figure}[ht]  
	\begin{center}
		\begin{tikzpicture}[scale=0.6]
			\draw[-,blue!30!black,thick] (-1,0) -- (11,0); 
			\draw[-,blue!30!black,thick] (1/2,-3) -- (11,1/2); 
			\draw[-,blue!30!black,thick] (2,-3) -- (2,3.8); 
			\draw[-,blue!30!black,thick] (1.5,15/4) -- (6,-3);
			--
			\filldraw[color=black, fill=blue!80!black] (0,0) circle (3pt)
			node[anchor=south west]{$P_1$};
			\filldraw[color=black, fill=blue!80!black] (2,0) circle (3pt)
			node[anchor=south west]{$P_2$};
			\filldraw[color=black, fill=blue!80!black] (4,0) circle (3pt)
			node[anchor=south west]{$P_3$};
			\filldraw[color=black, fill=blue!80!black] (6,0) circle (3pt)
			node[anchor=south west]{$P_4$};
			--				
			\filldraw[color=black, fill=blue!80!black] (2,3) circle (3pt)
			node[anchor=south west]{$Q_1$};
			\filldraw[color=black, fill=blue!80!black] (2,-2.5) circle (3pt);
			\filldraw[color=black, fill=blue!80!black] (5,-3/2) circle (3pt);
			--
			\node (Q2) at (2.4,-2) {$Q_2$};
			\node (Q3) at (5.3,-1) {$Q_3$};
			\node (L) at (10.5,-0.5) {$L$};
			\node (L1) at (0.2,-2.8) {$L_1$};
			\node (L2) at (2.4,-2.8) {$L_2$};
			\node (L3) at (6.5,-2.8) {$L_3$};
		\end{tikzpicture}\vspace*{-0.3cm}
	\end{center}
	\caption{Seven points with exactly two $\mathcal{Q}$-collinear points \label{Figure002}}
\end{figure}
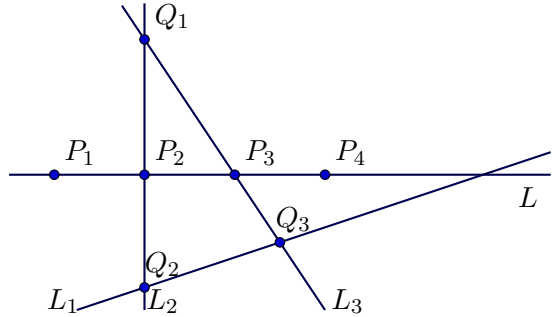

On the other hand, let $D$ be a divisor of degree $d$ vanishing 
along $m\mathbb{X}$. The B\'{e}zout decomposition of $D$ 
with respect to the set $\mathbb{X}$ and curves $L_1, L_2, L_3, L,C$ is
\[
D=kL_1+p(L_2+L_3)+qL+rC+B(D),
\]
where $B(D)$ is the B\'{e}zout reduction. 
Using Proposition~\ref{bezout decomposition}, we obtain
\begin{align*}
	B(D)\cdot L_1&=d-k-2p-q-2r \ge 2(m-k-p-r),\\
	B(D)\cdot L_2&=d-k-2p-q-2r \ge (m-2p-r)+(m-k-p-r)+(m-p-q),\\
	B(D)\cdot L &=d-k-2p-q-2r \ge 2(m-q-r)+2(m-p-q),\\
	B(D)\cdot C &=2(d-k-2p-q-2r)\ge 2(m-q-r)+2(m-k-p-r) +(m-2p-r).
\end{align*}
Together with the condition $B(D)\ge 0$, we get 
\begin{align}
	d&\ge k+2p+q+2r,  \quad \label{eq5}\\
	d&\ge 2m-k+q,  \quad \label{eq6}\\
	d& \ge 3m-2p,   \quad \label{eq7}\\
	d &\ge  4m+k-3q,   \quad \label{eq8}\\
	2d & \ge 5m-r.   \quad \label{eq9}
\end{align}
Multiplying the inequations $(\ref{eq5})$; $(\ref{eq6})$; $(\ref{eq7})$;
$(\ref{eq8})$ and $(\ref{eq9})$ by $1;$ $2$; $1$; $1$ and $2$, 
respectively, and then adding the resulting inequations side-by-side,
we obtain $d\ge \frac{7}{3}m.$ Hence we can conclude that $\hat{\alpha}_{\mathbb{X}}\ge \frac{7}{3}.$

\medskip\noindent (c)\quad 
Suppose that $P_4$ lie on $L_3$ (see Figure~\ref{Figure003}).
Let $C_1, C_2, C_3$ be the irreducible conics passing through 
$Q_1, Q_2, Q_3, P_2, P_3;$ $Q_1, Q_2, Q_3, P_1, P_3;$ 
and $Q_1, Q_2, Q_3, P_1, P_2,$ respectively.
The divisor $D'=m(C_1+C_2+C_3)+2m(L_1+L_2+L_3)+5mL$ is of degree $17m$ 
and vanishes to  order $7m$ at every point of $\mathbb{X}$. 
It follows that $\hat{\alpha}_{\mathbb{X}}\le \frac{17}{7}.$
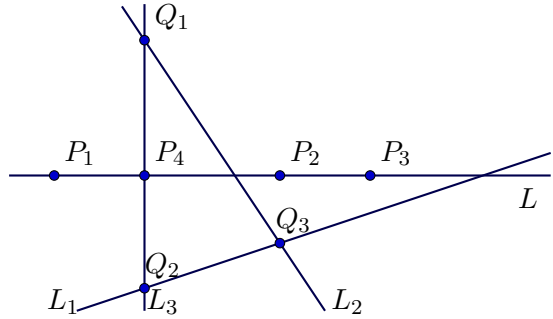
\begin{figure}[ht] 
	\begin{center}
		\begin{tikzpicture}[scale=0.6]
			\draw[-,blue!30!black,thick] (-1,0) -- (11,0); 
			\draw[-,blue!30!black,thick] (1/2,-3) -- (11,1/2); 
			\draw[-,blue!30!black,thick] (2,-3) -- (2,3.8); 
			\draw[-,blue!30!black,thick] (1.5,15/4) -- (6,-3);
			--
			\filldraw[color=black, fill=blue!80!black] (0,0) circle (3pt)
			node[anchor=south west]{$P_1$};
			\filldraw[color=black, fill=blue!80!black] (2,0) circle (3pt)
			node[anchor=south west]{$P_4$};
			\filldraw[color=black, fill=blue!80!black] (5,0) circle (3pt)
			node[anchor=south west]{$P_2$};
			\filldraw[color=black, fill=blue!80!black] (7,0) circle (3pt)
			node[anchor=south west]{$P_3$};
			--				
			\filldraw[color=black, fill=blue!80!black] (2,3) circle (3pt)
			node[anchor=south west]{$Q_1$};
			\filldraw[color=black, fill=blue!80!black] (2,-2.5) circle (3pt);
			\filldraw[color=black, fill=blue!80!black] (5,-3/2) circle (3pt);
			--
			\node (Q2) at (2.4,-2) {$Q_2$};
			\node (Q3) at (5.3,-1) {$Q_3$};
			\node (L) at (10.5,-0.5) {$L$};
			\node (L1) at (0.2,-2.8) {$L_1$};
			\node (L3) at (2.4,-2.8) {$L_3$};
			\node (L2) at (6.5,-2.8) {$L_2$};
		\end{tikzpicture}\vspace*{-0.3cm}
	\end{center}
	\caption{Seven points with only one $\mathcal{Q}$-collinear point
	\label{Figure003}}
\end{figure}

To prove the inverse side, let $D$ be a divisor of degree $d$ 
vanishing along $m\mathbb{X}$. The B\'{e}zout decomposition of $D$ with respect to the set $\mathbb{X}$ and curves $C_1, C_2, C_3, L_1, L_2, L_3$ and $L$ is 
$$
D=k(C_1+C_2+C_3)+pL_3+qL + r(L_1+L_2) +B(D).
$$
Applying Proposition~\ref{bezout decomposition}, we get
$$
\begin{aligned}
	B(D)\cdot C_1&= 2(d-6k-p-q-2r) \ge 2(m-3k-p-r)+(m-3k-2r)+2(m-2k-q),\\	
	B(D)\cdot L_1&=d-6k-p-q-2r \ge (m-3k-p-r)+(m-3k-2r),\\
	B(D)\cdot L_3 &=d-6k-p-q-2r \ge 2(m-3k-p-r)+(m-p-q),\\
	B(D)\cdot L&=d-6k-p-q-2r \ge 3(m-2k-q)+(m-p-q).
\end{aligned}
$$
Together with $B(D)\ge 0$, we get a system of inequations
\begin{align}
	d &\ge 6k+p+q+2r, \quad \label{eq10} \\
	2d& \ge 5m-k, \quad \label{eq11} \\
	d&\ge 2m+q-r, \quad \label{eq12} \\
	d&\ge 3m-2p, \quad \label{eq13} \\
	d& \ge 4m-3q+2r. \quad \label{eq14} 
\end{align}
Multiplying the inequations \eqref{eq10}; \eqref{eq11}; 
\eqref{eq12}; \eqref{eq13} and \eqref{eq14} by 
$2$; $12$;  $16$; $1$ and $6$, respectively, 
and then adding them side-by-side yield $d\ge \frac{17}{7}m.$ 
Hence we get $\hat{\alpha}_{\mathbb{X}}\ge \frac{17}{7}$, as desired.

\medskip\noindent (d)\quad 
Finally, consider the configuration of $\mathbb{X}$ as in 
Figure~\ref{Figure004}.
\begin{figure}[ht]  
	\begin{center}
		\begin{tikzpicture}[scale=0.6]
			\draw[-,blue!30!black,thick] (-1,0) -- (11,0); 
			\draw[-,blue!30!black,thick] (1/2,-3) -- (11,1/2); 
			\draw[-,blue!30!black,thick] (2,-3) -- (2,3.8); 
			\draw[-,blue!30!black,thick] (1.5,15/4) -- (6,-3);
			--
			\filldraw[color=black, fill=blue!80!black] (0,0) circle (3pt)
			node[anchor=south west]{$P_1$};
			\filldraw[color=black, fill=blue!80!black] (2.5,0) circle (3pt)
			node[anchor=south west]{$P_2$};
			\filldraw[color=black, fill=blue!80!black] (5,0) circle (3pt)
			node[anchor=south west]{$P_3$};
			\filldraw[color=black, fill=blue!80!black] (7,0) circle (3pt)
			node[anchor=south west]{$P_4$};
			--				
			\filldraw[color=black, fill=blue!80!black] (2,3) circle (3pt)
			node[anchor=south west]{$Q_1$};
			\filldraw[color=black, fill=blue!80!black] (2,-2.5) circle (3pt);
			\filldraw[color=black, fill=blue!80!black] (5,-3/2) circle (3pt);
			--
			\node (Q2) at (2.4,-2) {$Q_2$};
			\node (Q3) at (5.3,-1) {$Q_3$};
			\node (L) at (10.5,-0.5) {$L$};
			\node (L1) at (0.2,-2.8) {$L_1$};
			\node (L3) at (2.4,-2.8) {$L_3$};
			\node (L2) at (6.5,-2.8) {$L_2$};
		\end{tikzpicture}\vspace*{-0.3cm}
	\end{center}
	\caption{Seven points with no $\mathcal{Q}$-collinear points 
	\label{Figure004}}
\end{figure}
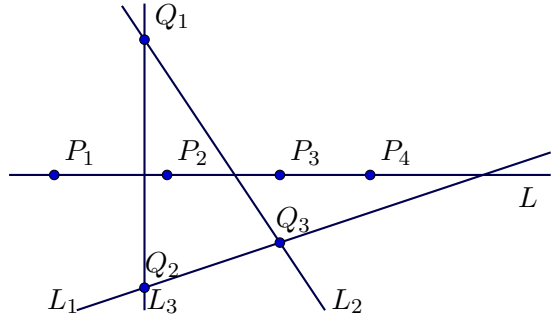
The divisor $D'=m(L_1+L_2+L_3)+2mL$ is of degree $5m$ 
and vanishes to order $2m$ at every point of $\mathbb{X}$. 
So, $\hat{\alpha}_{\mathbb{X}}\le \frac{5}{2}.$
Next, let $D$ be a divisor of degree $d$ vanishing along $m\mathbb{X}$. 
The B\'{e}zout decomposition of $D$ with respect to $\mathbb{X}$ and
the lines $L_1, L_2, L_3$ and $L$ is
$$
D=p(L_1+L_2+L_3)+qL+B(D).
$$ 
An application of Proposition~\ref{bezout decomposition} to the reduction $B(D)$ 
and $L$ or $L_1$ yields
\begin{align*}
	B(D)\cdot L &=d-3p-q \ge 4(m-q),\\
	B(D)\cdot L_1 &=d-3p-q\ge 2(m-2p),
\end{align*}
which implies
\begin{align}
	d&\ge 4m+3n-3t, \quad \label{eq017}\\
	d&\ge 2m-n+t. \quad \label{eq018}
\end{align}
Multiplying (\ref{eq017}) and (\ref{eq018}) by $1$ and $3$, respectively,
and then adding the resulting inequalities side-by-side yields 
$d\ge \frac{5}{2}m$. Therefore, we get 
$\hat{\alpha}_{\mathbb{X}}\ge \frac{5}{2}.$
\end{proof}

Below are some direct consequences of this Proposition.

\begin{corollary}\label{Cor-Waldschmidt-8Points-01}
Suppose that $\mathbb{X}=\{P_1,\dots,P_5,Q_1,Q_2,Q_3\}\subseteq \mathbb{P}^2$
and that three  points among $P_1,\dots, P_5$ are $\mathcal{Q}$-collinear. 
Then $\hat{\alpha}_{\mathbb{X}}=\frac{7}{3}.$ 
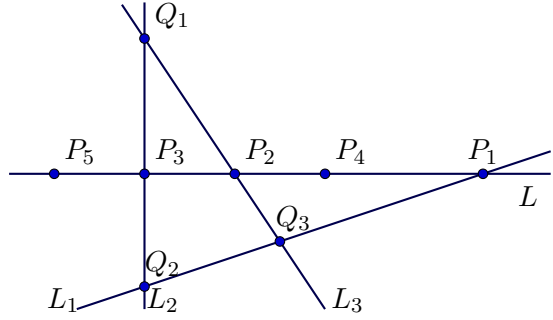
\begin{figure}[ht]  
		\begin{center}
			\begin{tikzpicture}[scale=0.6]
				\draw[-,blue!30!black,thick] (-1,0) -- (11,0); 
				\draw[-,blue!30!black,thick] (1/2,-3) -- (11,1/2); 
				\draw[-,blue!30!black,thick] (2,-3) -- (2,3.8); 
				\draw[-,blue!30!black,thick] (1.5,15/4) -- (6,-3);
				--
				\filldraw[color=black, fill=blue!80!black] (0,0) circle (3pt)
				node[anchor=south west]{$P_5$};
				\filldraw[color=black, fill=blue!80!black] (2,0) circle (3pt)
				node[anchor=south west]{$P_3$};
				\filldraw[color=black, fill=blue!80!black] (4,0) circle (3pt)
				node[anchor=south west]{$P_2$};
				\filldraw[color=black, fill=blue!80!black] (6,0) circle (3pt)
				node[anchor=south west]{$P_4$};
				\filldraw[color=black, fill=blue!80!black] (9.5,0) circle (3pt)
				node[anchor=south]{$P_1$};
				--				
				\filldraw[color=black, fill=blue!80!black] (2,3) circle (3pt)
				node[anchor=south west]{$Q_1$};
				\filldraw[color=black, fill=blue!80!black] (2,-2.5) circle (3pt);
				\filldraw[color=black, fill=blue!80!black] (5,-3/2) circle (3pt);
				--
				\node (Q2) at (2.4,-2) {$Q_2$};
				\node (Q3) at (5.3,-1) {$Q_3$};
				\node (L) at (10.5,-0.5) {$L$};
				\node (L1) at (0.2,-2.8) {$L_1$};
				\node (L2) at (2.4,-2.8) {$L_2$};
				\node (L3) at (6.5,-2.8) {$L_3$};
			\end{tikzpicture}\vspace*{-0.3cm}
		\end{center}
		\caption{Eight points with three $\mathcal{Q}$-collinear points}
\end{figure}
\end{corollary}
\begin{proof}
Assume that $P_i$ lies on $L_i$ for $i=1,2,3$. Let $C$ be the irreducible 
conic passing through $Q_1, Q_2, Q_3$, $P_4,P_5.$ Then
the divisor $D=m(L_1+L_2+L_3+C+2L)$ is of degree $7m$ and vanishes 
to order $3m$ at every point of $\mathbb{X}$. So, we get $\hat{\alpha}_{\mathbb{X}}\le \frac{7}{3}.$
On the other hand, consider the subset 
$\mathbb{Y}=\{P_2, P_3, P_4, P_5, Q_1, Q_2, Q_3\}$ of $\mathbb{X}$.
Clearly, $I_{\mathbb{X}}\subseteq I_{\mathbb{Y}}.$ Hence
$\hat{\alpha}_{\mathbb{X}}\ge \hat{\alpha}_{\mathbb{Y}}=\frac{7}{3}$ 
by  Proposition~\ref{Prop-WaldschmidtConst-7Points}(b). 
Thus $\hat{\alpha}_{\mathbb{X}}= \frac{7}{3}.$
\end{proof}

\begin{corollary}\label{Cor-Waldschmidt-89Points-02}
	Let $\mathbb{X}$ be the set of 7 points with configuration 
	given in Proposition~\ref{Prop-WaldschmidtConst-7Points}(c), 
	let $P_5$ (respectively $P_6$ ) be the intersection point 
	of $L_2$ (respectively $L_1$) and $L$, and let 
	$\mathbb{Y}=\mathbb{X}\cup \{P_5\}$ and $\mathbb{Z}=\mathbb{X}\cup\{P_5,P_6\}$
	(see Figure~\ref{Figure005}). 
	Then $
	\hat{\alpha}_{\mathbb{X}}=\hat{\alpha}_{\mathbb{Y}}
	=\hat{\alpha}_{\mathbb{Z}}=\frac{17}{7}.
	$
	\begin{figure}[ht] 
		\begin{center}
			\begin{tikzpicture}[scale=0.6]
				\draw[-,blue!30!black,thick] (-1,0) -- (11,0); 
				\draw[-,blue!30!black,thick] (1/2,-3) -- (11,1/2); 
				\draw[-,blue!30!black,thick] (2,-3) -- (2,3.8); 
				\draw[-,blue!30!black,thick] (1.5,15/4) -- (6,-3);
				--
				\filldraw[color=black, fill=blue!80!black] (0,0) circle (3pt)
				node[anchor=south west]{$P_1$};
				\filldraw[color=black, fill=blue!80!black] (2,0) circle (3pt)
				node[anchor=south west]{$P_4$};
				\filldraw[color=black, fill=blue!80!black] (4,0) circle (3pt)
				node[anchor=south west]{$P_5$};
				\filldraw[color=black, fill=blue!80!black] (5,0) circle (3pt)
				node[anchor=south west]{$P_2$};
				\filldraw[color=black, fill=blue!80!black] (7,0) circle (3pt)
				node[anchor=south west]{$P_3$};
				--				
				\filldraw[color=black, fill=blue!80!black] (2,3) circle (3pt)
				node[anchor=south west]{$Q_1$};
				\filldraw[color=black, fill=blue!80!black] (2,-2.5) circle (3pt);
				\filldraw[color=black, fill=blue!80!black] (5,-3/2) circle (3pt);
				--
				\node (Q2) at (2.4,-2) {$Q_2$};
				\node (Q3) at (5.3,-1) {$Q_3$};
				\node (L) at (10.5,-0.5) {$L$};
				\node (L1) at (0.2,-2.8) {$L_1$};
				\node (L3) at (2.4,-2.8) {$L_3$};
				\node (L2) at (6.5,-2.8) {$L_2$};
			\end{tikzpicture}\quad 
			\begin{tikzpicture}[scale=0.6]
				\draw[-,blue!30!black,thick] (-1,0) -- (11,0); 
				\draw[-,blue!30!black,thick] (1/2,-3) -- (11,1/2); 
				\draw[-,blue!30!black,thick] (2,-3) -- (2,3.8); 
				\draw[-,blue!30!black,thick] (1.5,15/4) -- (6,-3);
				--
				\filldraw[color=black, fill=blue!80!black] (0,0) circle (3pt)
				node[anchor=south west]{$P_1$};
				\filldraw[color=black, fill=blue!80!black] (2,0) circle (3pt)
				node[anchor=south west]{$P_4$};
				\filldraw[color=black, fill=blue!80!black] (4,0) circle (3pt)
				node[anchor=south west]{$P_5$};
				\filldraw[color=black, fill=blue!80!black] (5,0) circle (3pt)
				node[anchor=south west]{$P_2$};
				\filldraw[color=black, fill=blue!80!black] (7,0) circle (3pt)
				node[anchor=south west]{$P_3$};
				\filldraw[color=black, fill=blue!80!black] (9.5,0) circle (3pt)
				node[anchor=south]{$P_6$};
				--				
				\filldraw[color=black, fill=blue!80!black] (2,3) circle (3pt)
				node[anchor=south west]{$Q_1$};
				\filldraw[color=black, fill=blue!80!black] (2,-2.5) circle (3pt);
				\filldraw[color=black, fill=blue!80!black] (5,-3/2) circle (3pt);
				--
				\node (Q2) at (2.4,-2) {$Q_2$};
				\node (Q3) at (5.3,-1) {$Q_3$};
				\node (L) at (10.5,-0.5) {$L$};
				\node (L1) at (0.2,-2.8) {$L_1$};
				\node (L3) at (2.4,-2.8) {$L_3$};
				\node (L2) at (6.5,-2.8) {$L_2$};
			\end{tikzpicture}\vspace*{-0.3cm}
		\end{center}
		\caption{Eight or nine points with at least two 
			$\mathcal{Q}$-collinear points \label{Figure005}}
	\end{figure}
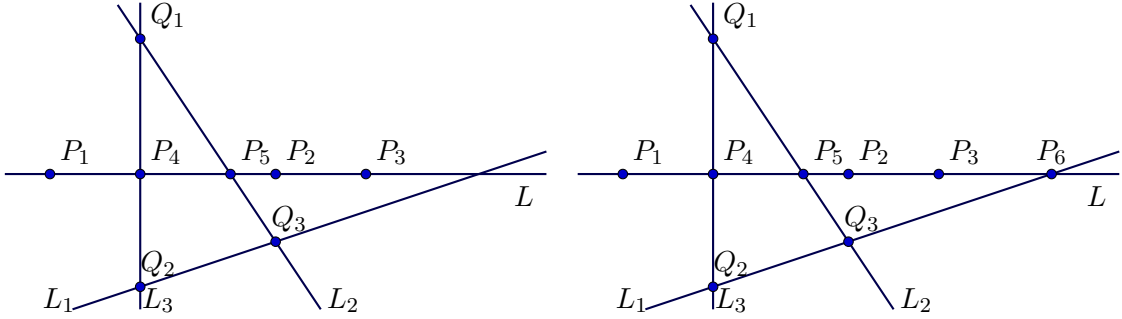
\end{corollary}	
\begin{proof}
	The inclusions of subschemes $\mathbb{X}\subseteq \mathbb{Y}\subseteq
	\mathbb{Z}$ and Proposition~\ref{Prop-WaldschmidtConst-7Points}(c) yield 
	$$
	\frac{17}{7}=\hat{\alpha}_\mathbb{X}\le 
	\hat{\alpha}_\mathbb{Y}\le  \hat{\alpha}_\mathbb{Z}.
	$$ 
    Moreover, let $C_1, C_2, C_3$ be the irreducible conics passing through 
	$Q_1, Q_2, Q_3, P_2, P_3;$ $Q_1, Q_2, Q_3, P_1, P_3;$ 
	and $Q_1, Q_2, Q_3, P_1, P_2,$ respectively.
	Then the divisor $D'=m(C_1+C_2+C_3)+2m(L_1+L_2+L_3)+5mL$ is of degree $17m$ 
	and vanishes to the order $7m$ at every point of $\mathbb{Z}$. 
	It follows that $\hat{\alpha}_{\mathbb{Z}}\le \frac{17}{7}.$
	Consequently, we have $\hat{\alpha}_{\mathbb{X}}=\hat{\alpha}_{\mathbb{Y}}
	=\hat{\alpha}_{\mathbb{Z}}= \frac{17}{7}$, as desired.
\end{proof}

Adding more points on the line $L$ in the configuration of part (d) of Proposition~\ref{Prop-WaldschmidtConst-7Points}, we obtain a class of configurations of $n>7$ points whose Waldschmidt constant is exactly $\dfrac{5}{2}$.

\begin{theorem}\label{Thm-WaldschmidtConst-nPoints}
	Let $\mathbb{Y}=\{P_1,\dots,P_4,Q_1,Q_2,Q_3\}$ be a set of 7 points 
	in $\mathbb{P}^2$ such that  $P_1,\dots, P_{4}$ lie on a line $L$,  
	$Q_1,Q_2,Q_3$ are out of $L$ and not collinear, and none of the points $P_1,\dots, P_{4}$ is $\mathcal{Q}$-collinear, and let 
	$\mathbb{X}$ be a set of $n> 7$ points obtained from $\mathbb{Y}$ 
	by extending $(n-7)$ additional points on $L$. Then
	$$
	\hat{\alpha}_{\mathbb{X}}=\hat{\alpha}_{\mathbb{Y}}=\frac{5}{2}.
	$$
	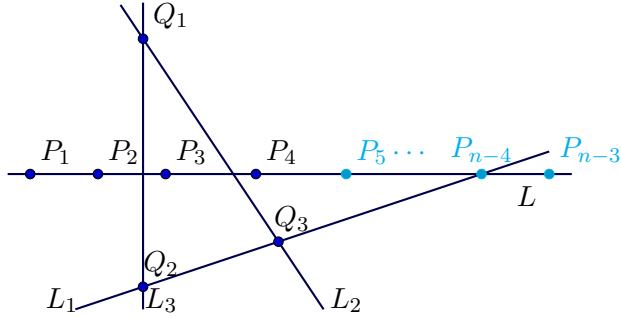
\begin{figure}[ht] 
		\begin{center}
			\begin{tikzpicture}[scale=0.6]
				\draw[-,blue!30!black,thick] (-1,0) -- (11.5,0); 
				\draw[-,blue!30!black,thick] (1/2,-3) -- (11,1/2); 
				\draw[-,blue!30!black,thick] (2,-3) -- (2,3.8); 
				\draw[-,blue!30!black,thick] (1.5,15/4) -- (6,-3);
				--
				\filldraw[color=black, fill=blue!80!black] (-0.5,0) circle (3pt)
				node[anchor=south west]{$P_1$};
				\filldraw[color=black, fill=blue!80!black] (1,0) circle (3pt)
				node[anchor=south west]{$P_2$};
				\filldraw[color=black, fill=blue!80!black] (2.5,0) circle (3pt)
				node[anchor=south west]{$P_3$};
				\filldraw[color=black, fill=blue!80!black] (4.5,0) circle (3pt)
				node[anchor=south west]{$P_4$};
				\filldraw[color=cyan, fill=cyan!80!black] (6.5,0) circle (3pt)
				node[anchor=south west]{$P_5\cdots$};
				\filldraw[color=cyan, fill=cyan!80!black] (11,0) circle (3pt)
				node[anchor=south west]{$P_{n-3}$};
				\filldraw[color=cyan, fill=cyan!80!black] (9.5,0) circle (3pt)
				node[anchor=south]{$P_{n-4}$};
				--				
				\filldraw[color=black, fill=blue!80!black] (2,3) circle (3pt)
				node[anchor=south west]{$Q_1$};
				\filldraw[color=black, fill=blue!80!black] (2,-2.5) circle (3pt);
				\filldraw[color=black, fill=blue!80!black] (5,-3/2) circle (3pt);
				--
				\node (Q2) at (2.4,-2) {$Q_2$};
				\node (Q3) at (5.3,-1) {$Q_3$};
				\node (L) at (10.5,-0.5) {$L$};
				\node (L1) at (0.2,-2.8) {$L_1$};
				\node (L3) at (2.4,-2.8) {$L_3$};
				\node (L2) at (6.5,-2.8) {$L_2$};
			\end{tikzpicture}\vspace*{-0.3cm}
		\end{center}
		\caption{$n$ points with at least four 
			non-$\mathcal{Q}$-collinear points}
	\end{figure}
\end{theorem}
\begin{proof}
Since $\mathbb{Y}$ is a subset of $\mathbb{X}$, it follows that
$\hat{\alpha}_{\mathbb{Y}}\le \hat{\alpha}_{\mathbb{X}}$.
By Proposition~\ref{Prop-WaldschmidtConst-7Points}(d), we have
$$
\frac{5}{2}=\hat{\alpha}_{\mathbb{Y}}\le \hat{\alpha}_{\mathbb{X}}.
$$
On the other hand, we consider the divisor $D = m(L_1+L_2+L_3)+2mL$.
Then $D$ is of degree $5m$ and vanishes to order $2m$ 
at every point of $\mathbb{X}$, and subsequently
$\hat{\alpha}_{\mathbb{X}}\le \frac{5}{2}.$ 
Therefore, we get $\hat{\alpha}_{\mathbb{X}}
=\hat{\alpha}_{\mathbb{Y}}=\frac{5}{2}.$
\end{proof}

The classification of all configurations of a set of $n$ points $\mathbb{X}$ 
in $\mathbb{P}^2$ whose  Waldschmidt constant satisfying
$\hat{\alpha}_{\mathbb{X}}\le 2$ is given by Main Theorem 
of~\cite{Mosakhani-Haghighi2016}. 
Combining the above results, the values of Waldschmidt constant of 
a set of $n\ge 7$ points, which has at least $n-3$ collinear points 
is either $1, \frac{2n-3}{n-1}, 2, \frac{16}{7}, \frac{7}{3}, \frac{17}{7}$ or $\frac{5}{2}.$
The following table shows us the relation between 
the configuration of points, the Hilbert polynomial 
and the Waldschmidt constant.

\noindent \renewcommand{\arraystretch}{1.5}
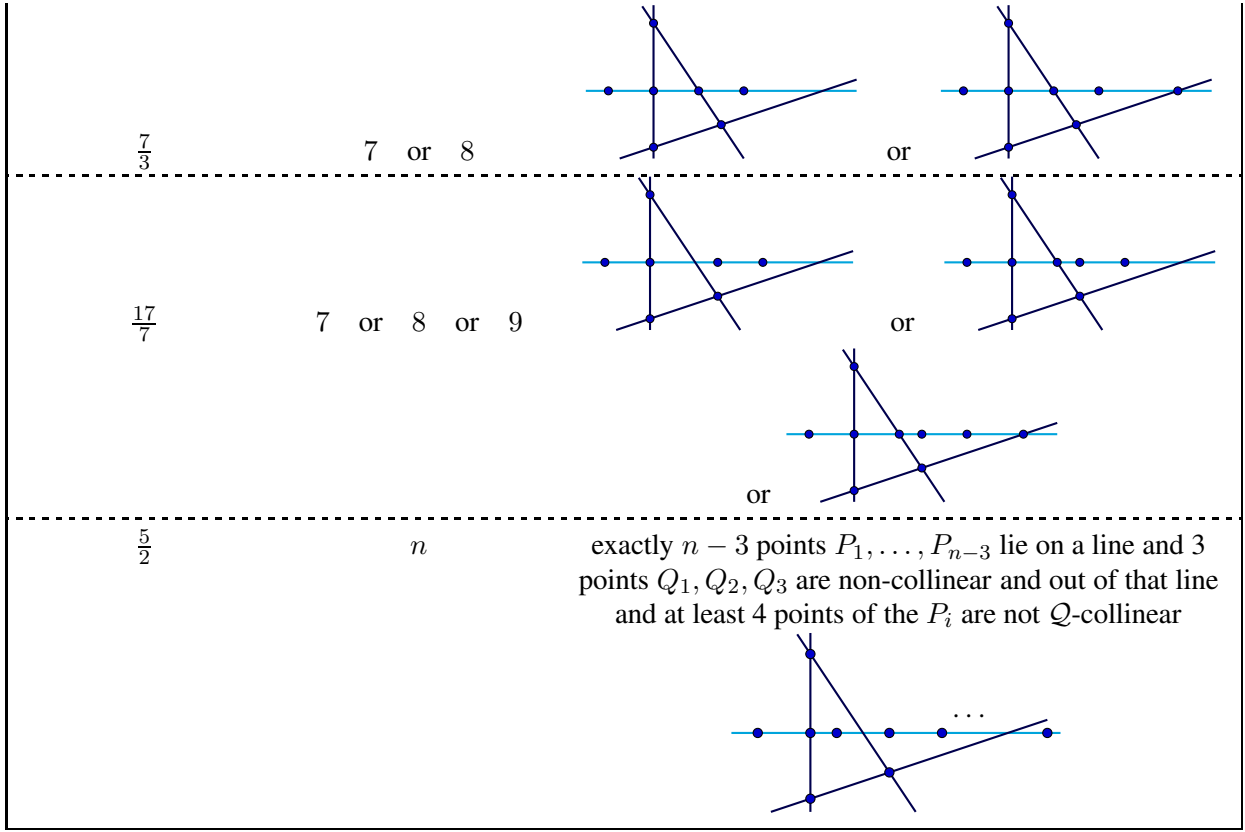
\begin{longtable}{|>{\centering\arraybackslash}p{0.2\textwidth} >{\centering\arraybackslash}p{0.2\textwidth} >{\centering\arraybackslash}p{0.53\textwidth}|}
	\hline
	\textbf{Waldschmidt Constant} & \textbf{Hilbert Polynomial} & \textbf{Configuration of Points} \\
	\hline
	1 & $n$ &  all $n$ points lie on a line \\ \cdashline{1-3}[2pt/3pt]
	$\frac{2n-3}{n-1}$ & $n$ &  exactly $n-1$ points are collinear 
	\\ \cdashline{1-3}[2pt/3pt]
	2 & $n$ &  exactly $n-2$ points are collinear 
	or exactly $n-3$ points lie on a line and 3 points are collinear and
	out of that line \\ \cdashline{1-3}[2pt/3pt] 
	$\frac{16}{7}$ & $7$ &  
	\begin{tikzpicture}[scale=0.3]
		\draw[-,blue!30!black,thick] (-1,0) -- (7.5,0);
		\draw[-,blue!30!black,thick] (0,-1) -- (0,5);
		\draw[-,blue!30!black,thick] (-0.5,29/6) -- (3,-1);
		\draw[-,cyan!90!black,thick] (-1,7/3) -- (7.5,-0.5);
		\filldraw[color=black, fill=blue!80!black] (0,0) circle (5pt)
		node[anchor=south west]{};
		\filldraw[color=black, fill=blue!80!black] (2.4,0) circle (5pt)
		node[anchor=south west]{};
		\filldraw[color=black, fill=blue!80!black] (0,4) circle (5pt)
		node[anchor=south west]{};
		\filldraw[color=black, fill=blue!80!black] (6,0) circle (5pt)
		node[anchor=south west]{};
		\filldraw[color=black, fill=blue!80!black] (0,2) circle (5pt)
		node[anchor=south west]{};
		\filldraw[color=black, fill=blue!80!black] (1.5,1.5) circle (5pt)
		node[anchor=south west]{};
		\filldraw[color=black, fill=blue!80!black] (4.5,0.5) circle (5pt)
		node[anchor=south west]{};
	\end{tikzpicture}
	\\ \cdashline{1-3}[2pt/3pt]
	$\frac{7}{3}$ & $7$\quad or\quad $8$ &  
	\begin{tikzpicture}[scale=0.3]
		\draw[-,cyan!90!black,thick] (-1,0) -- (11,0); 
		\draw[-,blue!30!black,thick] (1/2,-3) -- (11,1/2); 
		\draw[-,blue!30!black,thick] (2,-3) -- (2,3.8); 
		\draw[-,blue!30!black,thick] (1.5,15/4) -- (6,-3);
		--
		\filldraw[color=black, fill=blue!80!black] (0,0) circle (5pt)
		node[anchor=south west]{};
		\filldraw[color=black, fill=blue!80!black] (2,0) circle (5pt)
		node[anchor=south west]{};
		\filldraw[color=black, fill=blue!80!black] (4,0) circle (5pt)
		node[anchor=south west]{};
		\filldraw[color=black, fill=blue!80!black] (6,0) circle (5pt)
		node[anchor=south west]{};
		--				
		\filldraw[color=black, fill=blue!80!black] (2,3) circle (5pt)
		node[anchor=south west]{};
		\filldraw[color=black, fill=blue!80!black] (2,-2.5) circle (5pt);
		\filldraw[color=black, fill=blue!80!black] (5,-3/2) circle (5pt);
	\end{tikzpicture}\quad\mbox{or}\quad
	\begin{tikzpicture}[scale=0.3]
		\draw[-,cyan!90!black,thick] (-1,0) -- (11,0); 
		\draw[-,blue!30!black,thick] (1/2,-3) -- (11,1/2); 
		\draw[-,blue!30!black,thick] (2,-3) -- (2,3.8); 
		\draw[-,blue!30!black,thick] (1.5,15/4) -- (6,-3);
		--
		\filldraw[color=black, fill=blue!80!black] (0,0) circle (5pt)
		node[anchor=south west]{};
		\filldraw[color=black, fill=blue!80!black] (2,0) circle (5pt)
		node[anchor=south west]{};
		\filldraw[color=black, fill=blue!80!black] (4,0) circle (5pt)
		node[anchor=south west]{};
		\filldraw[color=black, fill=blue!80!black] (6,0) circle (5pt)
		node[anchor=south west]{};
		\filldraw[color=black, fill=blue!80!black] (9.5,0) circle (5pt)
		node[anchor=south]{};
		--				
		\filldraw[color=black, fill=blue!80!black] (2,3) circle (5pt)
		node[anchor=south west]{};
		\filldraw[color=black, fill=blue!80!black] (2,-2.5) circle (5pt);
		\filldraw[color=black, fill=blue!80!black] (5,-3/2) circle (5pt);
	\end{tikzpicture} 
	\\ \cdashline{1-3}[2pt/3pt]
	$\frac{17}{7}$ & $7$\quad or\quad $8$\quad or\quad $9$ &  
	\begin{tikzpicture}[scale=0.3]
		\draw[-,cyan!90!black,thick] (-1,0) -- (11,0); 
		\draw[-,blue!30!black,thick] (1/2,-3) -- (11,1/2); 
		\draw[-,blue!30!black,thick] (2,-3) -- (2,3.8); 
		\draw[-,blue!30!black,thick] (1.5,15/4) -- (6,-3);
		--
		\filldraw[color=black, fill=blue!80!black] (0,0) circle (5pt)
		node[anchor=south west]{};
		\filldraw[color=black, fill=blue!80!black] (2,0) circle (5pt)
		node[anchor=south west]{};
		\filldraw[color=black, fill=blue!80!black] (5,0) circle (5pt)
		node[anchor=south west]{};
		\filldraw[color=black, fill=blue!80!black] (7,0) circle (5pt)
		node[anchor=south west]{};
		--				
		\filldraw[color=black, fill=blue!80!black] (2,3) circle (5pt)
		node[anchor=south west]{};
		\filldraw[color=black, fill=blue!80!black] (2,-2.5) circle (5pt);
		\filldraw[color=black, fill=blue!80!black] (5,-3/2) circle (5pt);
	\end{tikzpicture} \quad\mbox{or}\quad
	\begin{tikzpicture}[scale=0.3]
		\draw[-,cyan!90!black,thick] (-1,0) -- (11,0); 
		\draw[-,blue!30!black,thick] (1/2,-3) -- (11,1/2); 
		\draw[-,blue!30!black,thick] (2,-3) -- (2,3.8); 
		\draw[-,blue!30!black,thick] (1.5,15/4) -- (6,-3);
		--
		\filldraw[color=black, fill=blue!80!black] (0,0) circle (5pt)
		node[anchor=south west]{};
		\filldraw[color=black, fill=blue!80!black] (2,0) circle (5pt)
		node[anchor=south west]{};
		\filldraw[color=black, fill=blue!80!black] (4,0) circle (5pt)
		node[anchor=south west]{};
		\filldraw[color=black, fill=blue!80!black] (5,0) circle (5pt)
		node[anchor=south west]{};
		\filldraw[color=black, fill=blue!80!black] (7,0) circle (5pt)
		node[anchor=south west]{};
		--				
		\filldraw[color=black, fill=blue!80!black] (2,3) circle (5pt)
		node[anchor=south west]{};
		\filldraw[color=black, fill=blue!80!black] (2,-2.5) circle (5pt);
		\filldraw[color=black, fill=blue!80!black] (5,-3/2) circle (5pt);
	\end{tikzpicture}
	\\
	&   &  
	\ \mbox{or} \ 
	\begin{tikzpicture}[scale=0.3]
		\draw[-,cyan!90!black,thick] (-1,0) -- (11,0); 
		\draw[-,blue!30!black,thick] (1/2,-3) -- (11,1/2); 
		\draw[-,blue!30!black,thick] (2,-3) -- (2,3.8); 
		\draw[-,blue!30!black,thick] (1.5,15/4) -- (6,-3);
		--
		\filldraw[color=black, fill=blue!80!black] (0,0) circle (5pt)
		node[anchor=south west]{};
		\filldraw[color=black, fill=blue!80!black] (2,0) circle (5pt)
		node[anchor=south west]{};
		\filldraw[color=black, fill=blue!80!black] (4,0) circle (5pt)
		node[anchor=south west]{};
		\filldraw[color=black, fill=blue!80!black] (5,0) circle (5pt)
		node[anchor=south west]{};
		\filldraw[color=black, fill=blue!80!black] (7,0) circle (5pt)
		node[anchor=south west]{};
		\filldraw[color=black, fill=blue!80!black] (9.5,0) circle (5pt)
		node[anchor=south]{};
		--				
		\filldraw[color=black, fill=blue!80!black] (2,3) circle (5pt)
		node[anchor=south west]{};
		\filldraw[color=black, fill=blue!80!black] (2,-2.5) circle (5pt);
		\filldraw[color=black, fill=blue!80!black] (5,-3/2) circle (5pt);
	\end{tikzpicture} 
	\\ \cdashline{1-3}[2pt/3pt]
	$\frac{5}{2}$ & $n$ &  exactly $n-3$ points $P_1,\dots,P_{n-3}$ lie on a line
	and 3 points $Q_1, Q_2, Q_3$ are non-collinear and out of that line and 
	at least 4 points of the $P_i$ are not $\mathcal{Q}$-collinear 
	
	\begin{tikzpicture}[scale=0.35]
		\draw[-,cyan!90!black,thick] (-1,0) -- (11.5,0); 
		\draw[-,blue!30!black,thick] (1/2,-3) -- (11,1/2); 
		\draw[-,blue!30!black,thick] (2,-3) -- (2,3.8); 
		\draw[-,blue!30!black,thick] (1.5,15/4) -- (6,-3);
		--
		\filldraw[color=black, fill=blue!80!black] (0,0) circle (5pt)
		node[anchor=south west]{};
		\filldraw[color=black, fill=blue!80!black] (2,0) circle (5pt)
		node[anchor=south west]{};
		\filldraw[color=black, fill=blue!80!black] (3,0) circle (5pt)
		node[anchor=south west]{};
		\filldraw[color=black, fill=blue!80!black] (5,0) circle (5pt)
		node[anchor=south west]{};
		\filldraw[color=black, fill=blue!80!black] (7,0) circle (5pt)
		node[anchor=south west]{$\cdots$};
		\filldraw[color=black, fill=blue!80!black] (11,0) circle (5pt)
		node[anchor=south west]{};
		--				
		\filldraw[color=black, fill=blue!80!black] (2,3) circle (5pt)
		node[anchor=south west]{};
		\filldraw[color=black, fill=blue!80!black] (2,-2.5) circle (5pt);
		\filldraw[color=black, fill=blue!80!black] (5,-3/2) circle (5pt);
	\end{tikzpicture}
	\\ 
	\hline
\end{longtable}

%
%
\goodbreak 
\section{Waldschmidt constant for sets of $n$ points with $n-1$ points on an irreducible conic}
In this section, we describe all configurations $\mathbb{X}$ of $n$ points with $n-1$ points among them lying on an irreducible conic such that  $\hat{\alpha}_{\mathbb{X}}\leq\dfrac{5}{2}$.

\subsection{Six points on an irreducible conic and one external point}\

In $\mathbb{P}^2$, let $C$ be an irreducible conic, let $P_1,\dots,P_6$
be points on $C$, and let $Q\notin C.$ Our aim is to determine the 
Waldschmidt constant $\hat{\alpha}_{\mathbb{X}}$ of the set of seven points 
$\mathbb{X}=\{P_1, P_2, \dots, P_6, Q\}$. 

An upper bound of $\hat{\alpha}_{\mathbb{X}}$ is given by the following lemma. 

\begin{lemma}\label{Sec3-lem}
In the setting above, we have $\hat{\alpha}_{\mathbb{X}} \le \frac{5}{2}$.
\end{lemma}

\begin{proof}
First of all we will prove that there exists a cubic $\mathfrak{C}$ 
passing through $P_1, P_2, \dots, P_6$ of multiplicity 1 and $Q$ 
of multiplicities 2. Indeed, let $\mathfrak{C}$ be a cubic of the form 
$$
\mathfrak{C}:
a_0x_0^3+a_1x_0^2x_1+ a_2x_0^2x_2+a_3x_0x_1^2+a_4x_0x_1x_2
+ a_5x_0x_2^2 +a_6x_1^3+a_7x_1^2x_2 +a_8x_1x_2^2+ a_9x_2^3=0,
$$
where $a_0,\dots,a_9\in \K$.
Then $P_i\in \mathfrak{C}$ if and only if $\mathfrak{C}(P_i)=0.$ 
Moreover, using Euler formula with a remark that $\charac(\K)=0$, 
 the point $Q\in \mathfrak{C}$ of multiplicity 2 if and only if 
$$
\frac{\partial \mathfrak{C}}{\partial x_0}(Q)
=\frac{\partial \mathfrak{C}}{\partial x_1}(Q)
=\frac{\partial \mathfrak{C}}{\partial x_2}(Q)=0.
$$
A homogeneous system of $9$ equations in $10$ variables always has 
a non-trivial solution. Hence, there is a cubic $\mathfrak{C}$ 
passing through $P_1, \dots, P_6$ of multiplicity $1$ 
and $Q$ of multiplicity 2.

For the irreducible conic $C$ passing through $P_1, \dots, P_6$, 
the divisor $m(\mathfrak{C}+C)$ is of degree $5m$ and vanishes to order $2m$ 
at every point of $\mathbb{X}.$ 
Therefore, we get $\hat{\alpha}_{\mathbb{X}}\le \frac{5}{2}.$
\end{proof}

In order to determine the Waldschmidt constant 
$\hat{\alpha}_{\mathbb{X}}$, we look closely at the 
following configurations of the set of points 
$\mathbb{X}=\{P_1,\dots,P_6,Q\}$:

\noindent \textbf{Type 1:}\ The external point $Q$ is an intersection of $3$ 
concurrent lines from pairs of points in $\{P_1,\dots,P_6\}$.

\noindent \textbf{Type 2:}\  Any three lines from pairs of points 
in $\{P_1,\dots,P_6\}$ are not concurrent or $Q$ is not the intersection point 
of any 3 concurrent lines from pairs of points in $\{P_1,\dots,P_6\}$. 

\begin{proposition}
\label{Prop 3.2 type 1 type 2}
With assumption as in Lemma \ref{Sec3-lem}, let $\hat{\alpha}_{\mathbb{X}}$
be the Waldschmidt constant of $\mathbb{X}$.
\begin{enumerate}
	\item[(a)] 
	$\hat{\alpha}_{\mathbb{X}}=\frac{7}{3}$ if and only if $\mathbb{X}$ is of Type 1.
	\item[(b)]  
	$\hat{\alpha}_{\mathbb{X}}=\frac{5}{2}$ if and only if $\mathbb{X}$ is of Type 2.
\end{enumerate}
\end{proposition}
\begin{proof}
Firstly, we will compute the Waldschmidt constant of $\mathbb{X}$ of Type 1
(see Figure \ref{Figure009}). 
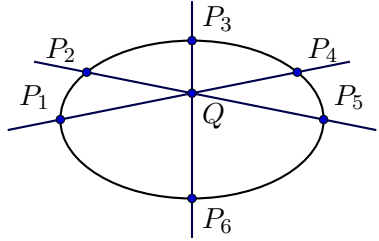
\begin{figure}[ht] 
	\begin{center}
		\begin{tikzpicture}[scale=0.35]
			\draw[-,blue!30!black,thick] (6,11/5) -- (-7,-2/5); 
			\draw[-,blue!30!black,thick] (-6,11/5) -- (7,-2/5); 
			\draw[-,blue!30!black,thick] (0,4.5) -- (0,-4.5); 
			--
			\draw[style=thick] (0,0) 	ellipse (5 and 3);
			--				
			\filldraw[color=black, fill=blue!80!black] (-5,0) circle (4.5pt)
			node[anchor=south east]{$P_1$};
			\filldraw[color=black, fill=blue!80!black] (4,9/5) circle (4.5pt)
			node[anchor=south west]{$P_4$};
			--
			\filldraw[color=black, fill=blue!80!black] (-4,9/5) circle (4.5pt)
			node[anchor=south east]{$P_2$};
			\filldraw[color=black, fill=blue!80!black] (5,0) circle (4.5pt)
			node[anchor=south west]{$P_5$};
			--
			\filldraw[color=black, fill=blue!80!black] (0,3) circle (4.5pt)
			node[anchor=south west]{$P_3$};
			\filldraw[color=black, fill=blue!80!black] (0,-3) circle (4.5pt)
			node[anchor=north west]{$P_6$};
			--
			\filldraw[color=black, fill=blue!80!black] (0,1) circle (4.5pt)
			node[anchor=north west]{$Q$};
		\end{tikzpicture}\vspace*{-0.5cm}
	\end{center}
	\caption{Six points on an irreducible conic and an external point - Type 1
		\label{Figure009}}
\end{figure}
Without lost of generally, we may assume that $Q=L_1\cap L_2\cap L_3$,
where $L_1=P_1P_4,  L_2=P_2P_5, L_3=P_3P_6$. Recalling that $C$ is the conic 
passing through $P_1, P_2, \dots, P_6.$ The divisor
$$D'=m(L_1+L_2+L_3)+2mC$$
is of degree $7m$ and contains $P_1, P_2, \dots, P_6, Q$ 
of multiplicity $3m.$ Thus we have 
$\hat{\alpha}_{\mathbb{X}}\le \frac{7}{3}.$

Now let $D$ be a divisor of degree $d$ containing $m\mathbb{X}$, and let 
$$
D \;=\; p(L_1+L_2+L_3)+qC+B(D)
$$ 
be the B\'{e}zout decomposition of $D$ with respect to $L_1, L_2, L_3$ and $C$. 
Applying Proposition~\ref{bezout decomposition} for $B(D)$ and $L_1, C$, we have
$$
\begin{aligned}
	B(D)\cdot L_1&= d-3p-2q \ge 2(m-p-q)+(m-3p),\\	
	B(D)\cdot C& =2(d-3p-2q) \ge 6(m-p-q).\\
\end{aligned}
$$
Together with $B(D)\ge 0$, we get a system of inequations:
\begin{align}
	d & \ge 3p+2q,  \quad \label{eq18}\\
	d &\ge 3m-2p,  \quad \label{eq19}\\	
	d & \ge 3m-q.  \quad \label{eq20}
\end{align}
Multiplying the inequations \eqref{eq18}, \eqref{eq19} and \eqref{eq20}
with respect to $2, 3$ and $4$ and summing them up yields
$d\ge \frac{7}{3}m$. 
Hence we have shown that $\hat{\alpha}_{\mathbb{X}}=\frac{7}{3}.$

Next, we compute $\hat{\alpha}_{\mathbb{X}}$ when $\mathbb{X}$ is of Type 2.
For this end, we distinguish the following subcases. 

\smallskip\noindent
\underline{\textbf{Subcase 1:}}\ 
\textit{None of the lines $P_iP_j$ contains $Q$ 
	(see Figure \ref{Figure010-1}).} \
\begin{figure}[ht] 
	\begin{center}
		\begin{tikzpicture}[scale=0.35]
			\draw[-,blue!30!black,thick] (-5,-18/5) -- (6,3); 
			\draw[-,blue!30!black,thick] (-6,11/5) -- (7,-2/5); 
			\draw[-,blue!30!black,thick] (0,4.2) -- (0,-4.2); 
			--
			\draw[style=thick] (0,0) 	ellipse (5 and 3);
			--
			\filldraw[color=black, fill=blue!80!black] (-4,9/5) circle (4.5pt)
			node[anchor=south east]{$P_1$};
			\filldraw[color=black, fill=blue!80!black] (5,0) circle (4.5pt)
			node[anchor=south west]{$P_2$};
			--
			\filldraw[color=black, fill=blue!80!black] (0,3) circle (4.5pt)
			node[anchor=south west]{$P_3$};
			\filldraw[color=black, fill=blue!80!black] (0,-3) circle (4.5pt)
			node[anchor=north west]{$P_4$};
			--				
			\filldraw[color=black, fill=blue!80!black] (-3,-12/5) circle (4.5pt)
			node[anchor=north ]{$P_5$};
			\filldraw[color=black, fill=blue!80!black] (4,9/5) circle (4.5pt)
			node[anchor=south]{$P_6$};
			--
			\filldraw[color=black, fill=blue!80!black] (-2.5,-1) circle (4.5pt)
			node[anchor=south west]{$Q$};
		\end{tikzpicture}\vspace*{-0.3cm}
	\end{center}
	\caption{Six points on an irreducible conic and an external point - Type 2(i)
		\label{Figure010-1}}
\end{figure}
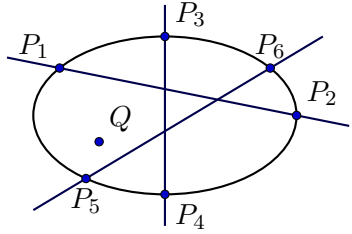
Let $C_1, C_2, C_3$ be the conics passing through $\{P_1, P_2, P_3, P_4, Q\},$ 
$\{P_3, P_4, P_5, P_6, Q\}$ and $\{P_1, P_2, P_5, P_6, Q\}$, respectively. 
Let $D=k(C_1+C_2+C_3)+pC+B(D)$ be the B\'{e}zout decomposition of a divisor $D$
of degree $d$ containing $m\mathbb{X}$. 
Applying Proposition~\ref{bezout decomposition} for $B(D)$ and $C_1$ we have
$$
B(D)\cdot C_1 =2(d-6k-2p) \ge 4(m-2k-p)+(m-3k).
$$
This induces an inequation $2d\ge 5m+k,$ and hence 
$\hat{\alpha}_{\mathbb{X}}\ge \frac{5}{2}.$ 
Thank to Lemma \ref{Sec3-lem}, the Waldschmidt constant 
of $\mathbb{X}$ is $\hat{\alpha}_{\mathbb{X}}= \frac{5}{2}.$ 

\smallskip\noindent
\underline{\textbf{Subcase 2:}}\  
\textit{There is only one line $P_iP_j$ containing $Q$ 
	(see Figure \ref{Figure010-2}).}\ 
\begin{figure}[ht] 
	\begin{center}
		\begin{tikzpicture}[scale=0.35]
			\draw[-,blue!30!black,thick] (-5,-18/5) -- (6,3); 
			\draw[-,blue!30!black,thick] (-6,11/5) -- (7,-2/5); 
			\draw[-,blue!30!black,thick] (0,4.2) -- (0,-4.2); 
			--
			\draw[style=thick] (0,0) 	ellipse (5 and 3);
			--
			\filldraw[color=black, fill=blue!80!black] (-4,9/5) circle (4.5pt)
			node[anchor=south east]{$P_1$};
			\filldraw[color=black, fill=blue!80!black] (5,0) circle (4.5pt)
			node[anchor=south west]{$P_2$};
			--
			\filldraw[color=black, fill=blue!80!black] (0,3) circle (4.5pt)
			node[anchor=south west]{$P_3$};
			\filldraw[color=black, fill=blue!80!black] (0,-3) circle (4.5pt)
			node[anchor=north west]{$P_4$};
			--				
			\filldraw[color=black, fill=blue!80!black] (-3,-12/5) circle (4.5pt)
			node[anchor=north ]{$P_5$};
			\filldraw[color=black, fill=blue!80!black] (4,9/5) circle (4.5pt)
			node[anchor=south]{$P_6$};
			--
			\filldraw[color=black, fill=blue!80!black] (-2,7/5) circle (4.5pt)
			node[anchor=south west]{$Q$};
		\end{tikzpicture}\vspace*{-0.3cm}
	\end{center}
	\caption{Six points on an irreducible conic and an external point - Type 2(ii)
		\label{Figure010-2}}
\end{figure}
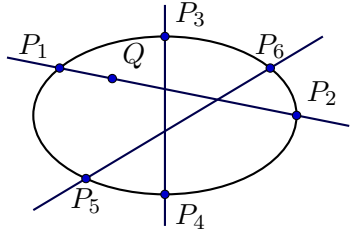
Without loss of generality, assume that $Q\in P_1P_2$.
Let $C_1$ be the irreducible conic defined by $\{P_3, P_4, P_5, P_6, Q\},$ 
let $L$ be the line  $P_1P_2$. Let $D$ be a divisor of degree $d$ containing $m\mathbb{X}$ and let 
$$
D=kL+pC_1+qC+B(D)
$$ 
be the B\'ezout decomposition of $D$ with respect to $L, C_1$ and $C$.
Applying Proposition~\ref{bezout decomposition} for $B(D)$ and $L, C_1, C$ we have
$$
\begin{aligned}
B(D)\cdot L & = d-k-2p-2q  \ge 2(m-k-q)+(m-k-p),\\
B(D)\cdot C_1 & = 2(d-k-2p-2q) \ge 4(m-p-q)+(m-k-p),\\
B(D)\cdot C & = 2(d-k-2p-2q) \ge 4(m-p-q)+2(m-k-q).
\end{aligned}
$$
Together with $B(D)\ge 0$, we get a system of inequations:
\begin{align}
d & \ge k+2p+2q, \quad \label{eq22}\\
d & \ge 3m-2k+p, \quad \label{eq23}\\
2d & \ge 5m+k-p,  \quad \label{eq24}\\
d & \ge 3m-q.  \quad \label{eq25}
\end{align}
Multiplying the inequations \eqref{eq22}, \eqref{eq23}, \eqref{eq24} 
and \eqref{eq25} by $1,$ $3,$ $5,$ and $2$ respectively, 
and then summing them up gives $d\ge \frac{5}{2}m$. 
Thank to Lemma \ref{Sec3-lem}, the Waldschmidt constant of 
$\mathbb{X}$ is $\frac{5}{2}.$

\smallskip\noindent 
\underline{\textbf{Subcase 3:}}\  
\textit{$Q$ is the intersection point of exactly 
	two lines of the form $P_iP_j$ (see Figure~\ref{Figure010-3}).}\
\begin{figure}[ht] 
	\begin{center}
		\begin{tikzpicture}[scale=0.35]
			\draw[-,blue!30!black,thick] (-5,-18/5) -- (6,3); 
			\draw[-,blue!30!black,thick] (-6,11/5) -- (7,-2/5); 
			\draw[-,blue!30!black,thick] (0,4.2) -- (0,-4.2); 
			--
			\draw[style=thick] (0,0) 	ellipse (5 and 3);
			--				
			\filldraw[color=black, fill=blue!80!black] (-4,9/5) circle (4.5pt)
			node[anchor=south east]{$P_1$};
			\filldraw[color=black, fill=blue!80!black] (5,0) circle (4.5pt)
			node[anchor=south west]{$P_2$};
			--
			\filldraw[color=black, fill=blue!80!black] (0,3) circle (4.5pt)
			node[anchor=south west]{$P_3$};
			\filldraw[color=black, fill=blue!80!black] (0,-3) circle (4.5pt)
			node[anchor=north west]{$P_4$};
			--
			\filldraw[color=black, fill=blue!80!black] (-3,-12/5) circle (4.5pt)
			node[anchor=north ]{$P_5$};
			\filldraw[color=black, fill=blue!80!black] (4,9/5) circle (4.5pt)
			node[anchor=south]{$P_6$};
			--
			\filldraw[color=black, fill=blue!80!black] (0,1) circle (4.5pt)
			node[anchor=south west]{$Q$};
		\end{tikzpicture}\vspace*{-0.3cm}
	\end{center}
	\caption{Six points on an irreducible conic and an external point - Type 2(iii)
		\label{Figure010-3}}
\end{figure}
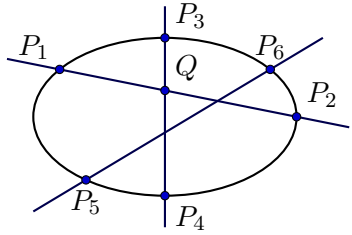
Without loss of generality, we may assume that 
$Q\in L_1\cap L_2$, where $L_1 = P_1P_2$ and $L_2 =P_3P_4$. 
Let $L_3=P_5Q, L_4=P_6Q.$ Let $D$ be a divisor of degree $d$ containing $m\mathbb{X}$, whose B\'ezout decomposition is
$$
D=k(L_1+L_2)+p(L_3+L_4)+qC+B(D).
$$
By Proposition~\ref{bezout decomposition}, we have
$$
\begin{aligned}
B(D)\cdot L_3 & = d-2k-2p-2q \ge (m-p-q)+(m-2k-2p),\\
B(D)\cdot C & = 2(d-2k-2p-2q) \ge 4(m-k-q)+2(m-p-q).
\end{aligned}
$$
Hence $d\ge 2m-p+q$ and $d\ge 3m-q+p$, these inequations imply 
$d\ge \frac{5}{2}m.$ Combining this with Lemma~\ref{Sec3-lem} 
yields $\hat{\alpha}_{\mathbb{X}} =\frac{5}{2}$, as wanted.
\end{proof}

\subsection{Configuration of seven points on an irreducible conic and one external point} \

As above, let $C$ be an irreducible conic in $\mathbb{P}^2$,
let $P_1,\dots, P_7$ be points on $C$ and $Q\notin C$, and let 
$\mathbb{X}=\{P_1, P_2, \dots, P_7, Q\}$ be the set of eight points. 
The Waldschmidt constant of $\mathbb{X}$ can be described as follows.

\begin{proposition} \label{Prop-8Points-ConicMinusOnePoint}
For $\mathbb{X}=\{P_1, P_2, \dots, P_7, Q\}\subseteq \mathbb{P}^2$ as above, 
we have $\hat{\alpha}_{\mathbb{X}}\ge \frac{5}{2}$. 
Furthermore, the equality holds if and only if $Q$ is an intersection point 
of three concurrent lines of the form $P_iP_j.$
\end{proposition}
\begin{proof}
If $Q$ is the intersection point of three concurrent lines of the form $P_iP_j$,
then,  w.l.o.g., we assume that $Q=L_1\cap L_2\cap L_3$, where 
$L_1=P_1P_2, L_2=P_3P_4, L_3=P_5P_6.$ By Proposition~\ref{Prop 3.2 type 1 type 2}, 
the subscheme $\mathbb{Y}=\mathbb{X}\setminus \{P_1\}$ of $\mathbb{X}$ has 
the Waldschmidt constant $\frac{5}{2}$, 
and thus $\hat{\alpha}_{\mathbb{X}}\ge \frac{5}{2}.$
Let $L$ be the line passing through $P_7$ and $Q$, then the divisor
$D=m(L_1+L_2+L_3+L+3C)$ contains $\mathbb{X}$ of multiplicity $4m$. Hence $\hat{\alpha}_{\mathbb{X}}\le \frac{5}{2}$. In conclusion, we have  $\hat{\alpha}_{\mathbb{X}}= \frac{5}{2}$.
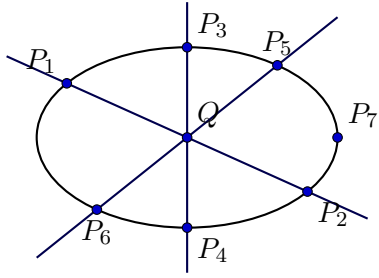
\begin{figure}[ht]
	\begin{center}
		\begin{tikzpicture}[scale=0.4]
			\draw[-,blue!30!black,thick] (5,4) -- (-5,-4); 
			\draw[-,blue!30!black,thick] (-6,54/20) -- (6,-54/20); 
			\draw[-,blue!30!black,thick] (0,4.5) -- (0,-4.5); 
			--
			\draw[style=thick] (0,0) 	ellipse (5 and 3);
			--				
			\filldraw[color=black, fill=blue!80!black] (5,0) circle (4.5pt)
			node[anchor=south west]{$P_7$};
			--
			\filldraw[color=black, fill=blue!80!black] (-4,9/5) circle (4.5pt)
			node[anchor=south east]{$P_1$};
			\filldraw[color=black, fill=blue!80!black] (4,-9/5) circle (4.5pt)
			node[anchor=north west]{$P_2$};
			--
			\filldraw[color=black, fill=blue!80!black] (0,3) circle (4.5pt)
			node[anchor=south west]{$P_3$};
			\filldraw[color=black, fill=blue!80!black] (0,-3) circle (4.5pt)
			node[anchor=north west]{$P_4$};
			--
			\filldraw[color=black, fill=blue!80!black] (3,12/5) circle (4.5pt)
			node[anchor=south]{$P_5$};
			\filldraw[color=black, fill=blue!80!black] (-3,-12/5) circle (4.5pt)
			node[anchor=north ]{$P_6$};
			--
			\filldraw[color=black, fill=blue!80!black] (0,0) circle (4.5pt)
			node[anchor=south west]{$Q$};
		\end{tikzpicture}\vspace*{-0.3cm}
	\end{center}
	\caption{Seven points on an irreducible conic and 
		a concurrently external point}
\end{figure}

Now we prove that, if there are not 3 lines of the form $P_iP_j$ concurrent at $Q$, then $\hat{\alpha}_{\mathbb{X}}\ge \frac{13}{5}> \frac{5}{2}.$
We have the following three circumstances.

\smallskip\noindent 
\underline{\textbf{Subcase 1:}}\ \textit{$Q$ is the intersection of 2 lines.}\\
Without loss of generality, we assume that $Q\in L_1\cap L_2$ where $L_1=P_1P_2$ and $L_2=P_3P_4.$ 
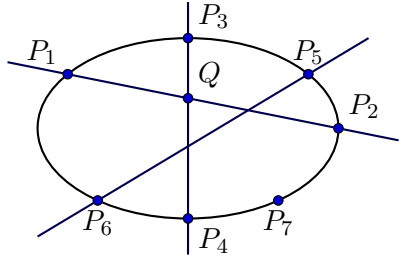
\begin{figure}[ht] 
	\begin{center}
		\begin{tikzpicture}[scale=0.4]
			\draw[-,blue!30!black,thick] (-5,-18/5) -- (6,3); 
			\draw[-,blue!30!black,thick] (-6,11/5) -- (7,-2/5); 
			\draw[-,blue!30!black,thick] (0,4.2) -- (0,-4.2); 
			--
			\draw[style=thick] (0,0) 	ellipse (5 and 3);
			--
			\filldraw[color=black, fill=blue!80!black] (-4,9/5) circle (4.5pt)
			node[anchor=south east]{$P_1$};
			\filldraw[color=black, fill=blue!80!black] (5,0) circle (4.5pt)
			node[anchor=south west]{$P_2$};
			--
			\filldraw[color=black, fill=blue!80!black] (0,3) circle (4.5pt)
			node[anchor=south west]{$P_3$};
			\filldraw[color=black, fill=blue!80!black] (0,-3) circle (4.5pt)
			node[anchor=north west]{$P_4$};
			--				
			\filldraw[color=black, fill=blue!80!black] (-3,-12/5) circle (4.5pt)
			node[anchor=north ]{$P_6$};
			\filldraw[color=black, fill=blue!80!black] (4,9/5) circle (4.5pt)
			node[anchor=south]{$P_5$};
			\filldraw[color=black, fill=blue!80!black] (3,-12/5) circle (4.5pt)
			node[anchor=north ]{$P_7$};
			--
			\filldraw[color=black, fill=blue!80!black] (0,1) circle (4.5pt)
			node[anchor=south west]{$Q$};
		\end{tikzpicture}\vspace*{-0.3cm}
	\end{center}
	\caption{Seven points on an irreducible conic and an external point - Subcase 1}
\end{figure}
Let $C_1, C_2$ be the irreducible conics defined by $\{P_3, P_5, P_6, P_7, Q\}$ and 
$\{P_4, P_5, P_6, P_7, Q\}$. 
Then $D'=m(C_1+C_2+L_1+2L_2+3C)$ vanishes along $\mathbb{X}$ 
of multiplicity $5m$, and thus $\hat{\alpha}_{\mathbb{X}}\le \frac{13}{5}.$
Now let $D$ be a divisor of degree $d$ vanishing along $m\mathbb{X}$. 
The B\'ezout decomposition of $D$ is given by
$$
D=k(C_1+C_2)+pL_1+qL_2+rC+B(D).
$$
By Proposition~\ref{bezout decomposition}, we have 
$$
\begin{aligned}
B(D)\cdot C_1&= 2(d-4k-p-q-2r) \ge 3(m-2k-r)+(m-k-q-r)+(m-2k-p-q),\\
B(D)\cdot L_1 & = d-4k-p-q-2r\ge 2(m-p-r)+(m-2k-p-q).
\end{aligned}
$$ 
This yields the system of inequalities
\begin{align}
 2d & \ge 5m-n+p, \quad \label{eq27}\\
 d & \ge 3m+2n-2p.\quad \label{eq28}
\end{align}
Multiplying the first  and the second inequality by $2$, $1$, respectively and summing up, we obtain $d\ge \frac{13}{5}m$, and so 
$\hat{\alpha}_{\mathbb{X}}\ge \frac{13}{5}.$
Consequently, we get $\hat{\alpha}_{\mathbb{X}}= \frac{13}{5}.$

\smallskip\noindent 
\underline{\textbf{Subcase 2:}}\ 
\textit{$Q$ lies on only one line of the form $P_iP_j$, says $L=P_6P_7.$ }
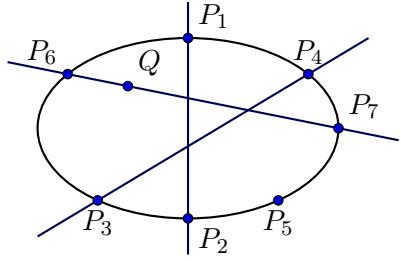
\begin{figure}[ht] 
	\begin{center}
		\begin{tikzpicture}[scale=0.4]
			\draw[-,blue!30!black,thick] (-5,-18/5) -- (6,3); 
			\draw[-,blue!30!black,thick] (-6,11/5) -- (7,-2/5); 
			\draw[-,blue!30!black,thick] (0,4.2) -- (0,-4.2); 
			--
			\draw[style=thick] (0,0) 	ellipse (5 and 3);
			--
			\filldraw[color=black, fill=blue!80!black] (-4,9/5) circle (4.5pt)
			node[anchor=south east]{$P_6$};
			\filldraw[color=black, fill=blue!80!black] (5,0) circle (4.5pt)
			node[anchor=south west]{$P_7$};
			--
			\filldraw[color=black, fill=blue!80!black] (0,3) circle (4.5pt)
			node[anchor=south west]{$P_1$};
			\filldraw[color=black, fill=blue!80!black] (0,-3) circle (4.5pt)
			node[anchor=north west]{$P_2$};
			--				
			\filldraw[color=black, fill=blue!80!black] (-3,-12/5) circle (4.5pt)
			node[anchor=north ]{$P_3$};
			\filldraw[color=black, fill=blue!80!black] (4,9/5) circle (4.5pt)
			node[anchor=south]{$P_4$};
			\filldraw[color=black, fill=blue!80!black] (3,-12/5) circle (4.5pt)
			node[anchor=north ]{$P_5$};
			--
			\filldraw[color=black, fill=blue!80!black] (-2,7/5) circle (4.5pt)
			node[anchor=south west]{$Q$};
		\end{tikzpicture}\vspace*{-0.3cm}
	\end{center}
	\caption{Seven points on an irreducible conic and an external point - Subcase 2}
\end{figure}

\noindent 
Thank to Lemma \ref{Sec3-lem}, let $\mathfrak{C}_1, \mathfrak{C}_2$ 
be irreducible cubic curves passing sets of 7 points 
$\{P_1, P_2, \dots, P_5, P_6, Q\}$ and 
$\{P_1, P_2, \dots, P_5, P_7, Q\}$, respectively.
Note that these cubic curves pass through $Q$ with the multiplicity 2. 
Observe that the divisor $D'=\mathfrak{C}_1+\mathfrak{C}_2+L+3C$ 
is of degree $13$ and vanishes along $5\mathbb{X}$. This yields
$\hat{\alpha}_{\mathbb{X}}\le \frac{13}{5}.$
Now we will prove that $\hat{\alpha}_{\mathbb{X}}=\frac{13}{5}.$ 
Let $C_i$ be the irreducible conic passing through 
$\{Q, P_1, P_2, P_3, P_4, P_5\}\setminus \{P_i\}$ for $i=1,\dots,5$,
and let $D$ be a divisor of degree $d$ passing through $m\mathbb{X}$.
The B\'ezout decomposition of $D$ is given by
$$
D = k(C_1+\cdots+C_5)+pL+qC + B(D).
$$
An application of Proposition~\ref{bezout decomposition} implies
$$
\begin{aligned}
B(D)\cdot C_1 &= 2(d-10k-p-2q) \ge 4(m-4k-q)+(m-5k-p),\\
B(D) \cdot L &= d-10k-p-2q  \ge 2(m-p-q)+(m-5k-p),
\end{aligned}
$$
and subsequently 
\begin{align*}
2d & \ge 5m-k+p, \quad \\
d & \ge 3m +5k-2p.
\end{align*}
As in the first subcase, these inequalities yield  $d\ge \frac{13}{5}m$, 
and so $\hat{\alpha}_{\mathbb{X}}\ge\frac{13}{5}$.
Hence we conclude that $\hat{\alpha}_{\mathbb{X}}=\frac{13}{5}.$

\smallskip\noindent 
\underline{\textbf{Subcase 3:}}\  
\textit{$Q$ does not lie on any line of the form $P_iP_j.$} 
\begin{figure}[ht] 
	\begin{center}
		\begin{tikzpicture}[scale=0.4]
			\draw[-,blue!30!black,thick] (-5,-18/5) -- (6,3); 
			\draw[-,blue!30!black,thick] (-6,11/5) -- (7,-2/5); 
			\draw[-,blue!30!black,thick] (0,4.2) -- (0,-4.2); 
			--
			\draw[style=thick] (0,0) 	ellipse (5 and 3);
			--
			\filldraw[color=black, fill=blue!80!black] (-4,9/5) circle (4.5pt)
			node[anchor=south east]{$P_5$};
			\filldraw[color=black, fill=blue!80!black] (5,0) circle (4.5pt)
			node[anchor=south west]{$P_6$};
			--
			\filldraw[color=black, fill=blue!80!black] (0,3) circle (4.5pt)
			node[anchor=south west]{$P_1$};
			\filldraw[color=black, fill=blue!80!black] (0,-3) circle (4.5pt)
			node[anchor=north west]{$P_2$};
			--				
			\filldraw[color=black, fill=blue!80!black] (-3,-12/5) circle (4.5pt)
			node[anchor=north ]{$P_3$};
			\filldraw[color=black, fill=blue!80!black] (4,9/5) circle (4.5pt)
			node[anchor=south]{$P_4$};
			\filldraw[color=black, fill=blue!80!black] (3,-12/5) circle (4.5pt)
			node[anchor=north ]{$P_7$};
			--
			\filldraw[color=black, fill=blue!80!black] (0.5,0.3) circle (4.5pt)
			node[anchor=south west]{$Q$};
		\end{tikzpicture}\vspace*{-0.3cm}
	\end{center}
	\caption{Seven points on an irreducible conic and an external point - Subcase 3}
\end{figure}
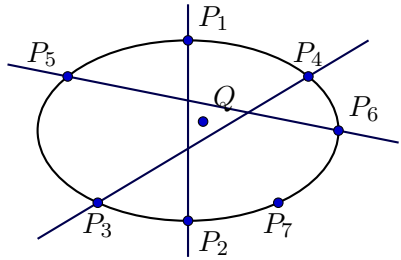

\noindent
Let $C_1$ be the irreducible conic passing through $\{Q, P_4, P_5, P_6, P_7\}$ 
and let $L_i=P_iQ$ for $1\le i\le 3.$ Further, let $D$ be a divisor 
of degree $d$ passing through $m\mathbb{X}$. The B\'ezout decomposition 
of $D$ for $C_1, L_1, L_2, L_3$ and $C$ is given by
$$
D = kC_1+p(L_1+L_2+L_3)+qC+B(D).
$$
By Proposition~\ref{bezout decomposition}, we get
$$
\begin{aligned}
B(D)\cdot L_1& = d-2k-3p-2q \ge (m-p-q)+(m-k-3p),\\
B(D) \cdot C & = 2(d-2k-3p-2q) \ge 4(m-k-q)+3(m-p-q).
\end{aligned}
$$
This implies $d  \ge 2m+k-p+q$ and $2d  \ge 7m+3p-3q$, and hence
$5d \ge 13m+3k$. Therefore, we get 
the inequality $\hat{\alpha}_{\mathbb{X}}\ge \frac{13}{5}.$
\end{proof}

From the proposition, we derive the following consequence.

\begin{corollary}\label{Sec3-Coro}
Let $\mathbb{X}=\{P_1,\dots, P_{n-1}, Q\}$ be a set of $n\ge 8$ points 
in $\mathbb{P}^2$, where $P_1,\dots, P_{n-1}$ are contained in 
an irreducible conic $C$ and $Q\notin C.$ Then
\begin{enumerate}
	\item[(a)] $\hat{\alpha}_{\mathbb{X}}\ge \frac{5}{2}$.
	\item[(b)] $\hat{\alpha}_{\mathbb{X}}=\frac{5}{2}$ if and only if
	one of the following conditions is satisfied:
	\begin{enumerate}
		\item[(i)] $n=8$ and $Q$ is the intersection point of three concurrent lines of the form $P_iP_j$.
		\item[(ii)] $n=9$ and $Q$ is the intersection point of four concurrent lines of the form $P_iP_j$.
	\end{enumerate}
\end{enumerate}
\end{corollary}
\begin{proof}
(a)\quad 
Let $\mathbb{Y}=\{P_1, P_2, \dots, P_7, Q\}\subseteq \mathbb{X}$.
Proposition~\ref{Prop-8Points-ConicMinusOnePoint} yields that 
$\hat{\alpha}_{\mathbb{X}}\ge \hat{\alpha}_{\mathbb{Y}}\ge \frac{5}{2}$.

\medskip\noindent (b)\quad 
If $n=8$ and $Q$ is the intersection point of three concurrent lines 
of the form $P_iP_j$, then Proposition~\ref{Prop-8Points-ConicMinusOnePoint} 
implies $\hat{\alpha}_{\mathbb{X}} = \frac{5}{2}$.
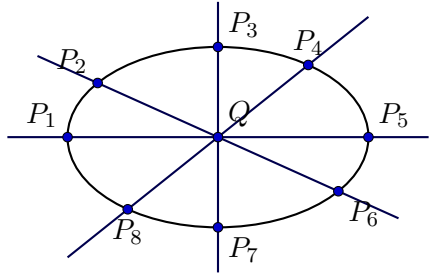
\begin{figure}[ht]
	\begin{center}
		\begin{tikzpicture}[scale=0.4]
			\draw[-,blue!30!black,thick] (5,4) -- (-5,-4); 
			\draw[-,blue!30!black,thick] (-6,54/20) -- (6,-54/20); 
			\draw[-,blue!30!black,thick] (0,4.5) -- (0,-4.5); 
			\draw[-,blue!30!black,thick] (-7,0) -- (7,0); 
			--
			\draw[style=thick] (0,0) 	ellipse (5 and 3);
			--				
			\filldraw[color=black, fill=blue!80!black] (-5,0) circle (4.5pt)
			node[anchor=south east]{$P_1$};
			\filldraw[color=black, fill=blue!80!black] (5,0) circle (4.5pt)
			node[anchor=south west]{$P_5$};
			--
			\filldraw[color=black, fill=blue!80!black] (-4,9/5) circle (4.5pt)
			node[anchor=south east]{$P_2$};
			\filldraw[color=black, fill=blue!80!black] (4,-9/5) circle (4.5pt)
			node[anchor=north west]{$P_6$};
			--
			\filldraw[color=black, fill=blue!80!black] (0,3) circle (4.5pt)
			node[anchor=south west]{$P_3$};
			\filldraw[color=black, fill=blue!80!black] (0,-3) circle (4.5pt)
			node[anchor=north west]{$P_7$};
			--
			\filldraw[color=black, fill=blue!80!black] (3,12/5) circle (4.5pt)
			node[anchor=south]{$P_4$};
			\filldraw[color=black, fill=blue!80!black] (-3,-12/5) circle (4.5pt)
			node[anchor=north ]{$P_8$};
			--
			\filldraw[color=black, fill=blue!80!black] (0,0) circle (4.5pt)
			node[anchor=south west]{$Q$};
		\end{tikzpicture}\vspace*{-0.3cm}
	\end{center}
	\caption{Eight points on an irreducible conic and 
		a concurrently external point \label{Figure11}}
\end{figure}
Consider the case that $n=9$ and $Q$ is the intersection point 
of four concurrent lines of the form $P_iP_j$ (see Figure~\ref{Figure11}).
By (a), we have $\hat{\alpha}_{\mathbb{X}} \ge \frac{5}{2}$.
Without loss of generality, we may assume that $Q=\cap_{i=1}^4 L_i$ 
with $L_i = P_iP_{4+i}$. Then the divisor $D= m(L_1+L_2+L_3+L_4+3C)$
is of degree $10$ and passes through $\mathbb{X}$ of multiplicity $4m$.
It follows that $\hat{\alpha}_{\mathbb{X}} \le \frac{5}{2}$.
So, in this case we also have $\hat{\alpha}_{\mathbb{X}} = \frac{5}{2}$.

Next we consider the case that $\mathbb{X}$ does not have 
the configurations as in (i) and (ii). Then $\mathbb{X}$ contains a subset
$\mathbb{Y}$ of $8$ points including $Q$ such that $Q$ is
not an intersection point of any three concurrent lines of the form $P_iP_j$.
By Proposition~\ref{Prop-8Points-ConicMinusOnePoint}, 
we have $\hat{\alpha}_{\mathbb{X}}\ge \hat{\alpha}_{\mathbb{Y}}
\ge \frac{13}{5} > \frac{5}{2}$. This completes the proof.
\end{proof}

%
%
\goodbreak 
\section{Some configurations  of $9$ points whose Waldschmidt constant equals to $\frac{5}{2}$}

In this section, we let $\mathbb{X}=\{P_1,\dots,P_9\}$ be a set of nine points
in $\mathbb{P}^2$. We would like to see whether the Waldschmidt constant 
equals to $\frac{5}{2}$. It is well-known that there is a cubic curve 
passing through $\mathbb{X}$. When the cubic curve containing 
$\mathbb{X}$ is irreducible, we get the following observation.

\begin{proposition}
	If $\mathbb{X}$ is contained in an irreducible cubic curve $\mathfrak{C}$, 
	then $\hat{\alpha}_{\mathbb{X}}=3.$
\end{proposition}
\begin{proof} 
Observe that $m\mathfrak{C}$ is a curve passing through $m\mathbb{X}$. 
Hence, the Waldschmidt constant of $\mathbb{X}$ is bounded by
$
\hat{\alpha}_{\mathbb{X}}\le 3.
$
Now let $D$ be a divisor of degree $d$ passing through $m\mathbb{X}$. 
The B\'{e}zout decomposition of $D$ via $\mathfrak{C}$ is defined by
$$
D = k\mathfrak{C}+B(D).
$$
Hence, we have
$$
B(D)\cdot \mathfrak{C} =3(d-3k) \ge 9(m-k).
$$ 
This inequality implies that 
$\hat{\alpha}_{\mathbb{X}}\ge 3$ and hence, we get $\hat{\alpha}_{\mathbb{X}}=3,$ as desired.
\end{proof}

Based on this observation, we have $\hat{\alpha}_{\mathbb{X}}=\frac{5}{2}$
only if $\mathbb{X}$ is not contained in an irreducible cubic curve.
In particular, any set of $n\ge 9$ points in $\mathbb{P}^2$ has
the Waldschmidt constant equal to $\frac{5}{2}$ only if there is no subset consisting
of $9$ points contained in an irreducible cubic.
When $\hat{\alpha}_{\mathbb{X}} < 3$, the set $\mathbb{X}$ 
must lie on 3 lines or on the union of an irreducible conic 
and one line. In this paper, we restrict our discussion on the second case.

Let $\mathbb{X}=\{P_1,\dots,P_9\}$ be a set of nine points 
in $\mathbb{P}^2$ that is not contained  in any irreducible cubic curve. 
Suppose that the points of $\mathbb{X}$ lie in an irreducible conic $C$ 
and a line.  We will consider the following circumstances 
for the position of points in $\mathbb{X}$.

\begin{enumerate}
	\item Eight points $P_1,\dots,P_8$ lie on $C$ and the remaining point 
	$P_9\notin C.$
	\item Seven points $P_1,\dots,P_7$ lie on $C$ and two others  
	$P_8, P_9 \notin C.$
	\item Six points $P_1,\dots,P_6$ lie on $C$ and three others 
	$P_7,P_8,P_9$ lie on a line $L$.
	\item Five points $P_1,\dots,P_5$ lie on $C$ and four others 
	$P_6,\dots,P_9$ lie on a line $L$.
\end{enumerate}

When the case (1) occurs, by Corollary~\ref{Sec3-Coro}, we have 
$\hat{\alpha}_{\mathbb{X}}\ge \frac{5}{2}$. The equality holds 
if and only if $Q$ is the intersection point of four concurrent lines 
of the form $P_iP_j$.

Now let us consider the case (2). 
In this case the Waldschmidt constant is estimated as follows.

\begin{proposition}
If $\mathbb{X}$ has the configuration as in (2), 
then $\hat{\alpha}_{\mathbb{X}}\ge \frac{18}{7}>\frac{5}{2}$. 
\end{proposition}
\begin{proof}
Observe that $\hat{\alpha}_{\mathbb{X}}\ge \frac{5}{2}$
by Corollary \ref{Sec3-Coro}.
If $P_8$ or $P_9$ is not a point of concurrency of a triple of lines 
$P_iP_j$ with $1\leq i,j\leq 7, i\not=j$, then  
$\{P_1,\dots,P_7, P_8\}$ or $\{P_1,\dots,P_7, P_9\}$ has 
the Waldschmidt constant greater than or equal to $\frac{13}{5}$ 
by Proposition~\ref{Prop-8Points-ConicMinusOnePoint}.
In this case, $\hat{\alpha}_\mathbb{X}\ge \frac{13}{5} > \frac{18}{7}$. 
Now we suppose that both $P_8, P_9$ are the points 
of concurrency of two triples of lines $P_iP_j$ 
with $1\leq i,j\leq 7, i\not=j$. Then we will show that
$\hat{\alpha}_\mathbb{X}\ge \frac{18}{7}$ 
by distinguishing two following subcases. 

\smallskip\noindent 
\underline{\textbf{Subcase 1:}}\  
\textit{The two triples of concurrent lines have a common line.} 
Without loss of generality, we may assume that $\mathbb{X}=\{P_1,\dots,P_9\}$
has the configuration as in Figure~\ref{Figure16}.
Let $L_{ij} = P_iP_j$ for $1\le i<j\le 9$ and let $D$ be a divisor 
of degree $d$ passing through $m\mathbb{X}$. 
\begin{figure}[ht] 
	\begin{center}
		\begin{tikzpicture}[scale=0.4]
			\draw[-,blue!30!black,thick] (6,11/5) -- (-7,-2/5); 
			\draw[-,blue!30!black,thick] (-6,11/5) -- (7,-2/5); 
			\draw[-,blue!30!black,thick] (0,4.5) -- (0,-4.5); 
			\draw[-,blue!30!black,thick] (6,21/5) -- (-1.5,-24/5); 
			\draw[-,blue!30!black,thick] (6,-12/5) -- (-1.5,87/20); 
			--
			\draw[style=thick] (0,0) 	ellipse (5 and 3);
			--				
			\filldraw[color=black, fill=blue!80!black] (-5,0) circle (4.5pt)
			node[anchor=south east]{$P_3$};
			\filldraw[color=black, fill=blue!80!black] (4,9/5) circle (4.5pt)
			node[anchor=south]{$P_4$};
			--
			\filldraw[color=black, fill=blue!80!black] (-4,9/5) circle (4.5pt)
			node[anchor=south east]{$P_1$};
			\filldraw[color=black, fill=blue!80!black] (5,0) circle (4.5pt)
			node[anchor=south west]{$P_2$};
			--
			\filldraw[color=black, fill=blue!80!black] (0,3) circle (4.5pt)
			node[anchor=south west]{$P_5$};
			\filldraw[color=black, fill=blue!80!black] (0,-3) circle (4.5pt)
			node[anchor=north west]{$P_6$};
			--
			\filldraw[color=black, fill=blue!80!black] (0,1) circle (4.5pt)
			node[anchor=north east]{$P_8$};
			\filldraw[color=black, fill=blue!80!black] (20/7,3/7) circle (4.5pt)
			node[anchor=north]{$P_9$};
			\filldraw[color=black, fill=blue!80!black] (9.2/2,-57/50) circle (4.5pt)
			node[anchor=north]{$P_7$};
			\draw[color=black] (1,-5.4) node {(i)};
		\end{tikzpicture}\qquad\quad  
		\begin{tikzpicture}[scale=0.5]
					\draw [rotate around={0:(1,1)},style=thick] (1,1) ellipse (3.6055512754639887cm and 2cm);
					\draw [style=thick,domain=-4.5:8] plot(\x,{(--6-0*\x)/6});
					\draw [style=thick,domain=-2:8] plot(\x,{(-3.9846050478594517--0.931986882737109*\x)/3.122869641295142});
					\draw [style=thick,domain=-2:8] plot(\x,{(--21.97118556316587-1.9992549245813636*\x)/6.724561348994481});
					\draw [style=thick,domain=-0.7:6.5] plot(\x,{(--12.204629322521468-3.0668471331209055*\x)/3.147310038825793});
					\draw [style=thick,domain=-0.7:6.5] plot(\x,{(--5.963119519166904-3.061058903355351*\x)/-3.129797812939109});
					--				
					\filldraw[color=black, fill=blue!80!black] 
					(0.9015917877116715,2.9992549245813636) circle (3.5pt)
					node[anchor=south west]{$P_3$};
					\filldraw[color=black, fill=blue!80!black] 
					(4.055829998181412,2.061479812078704) circle (3.5pt)
					node[anchor=south]{$P_4$};
					--
					\filldraw[color=black, fill=blue!80!black] 
					(-2.6055512754639887,1) circle (3.5pt)
					node[anchor=south east]{$P_1$};
					\filldraw[color=black, fill=blue!80!black] 
					(4.60555127546399,1) circle (3.5pt)
					node[anchor=south west]{$P_2$};
					--
					\filldraw[color=black, fill=blue!80!black] 
					(4.04890182653745,-0.0675922085395424) circle (3.5pt)
					node[anchor=north]{$P_5$};
					\filldraw[color=black, fill=blue!80!black] 
					(0.9260321852423081,-0.9995790912766513) circle (3.5pt)
					node[anchor=north west]{$P_6$};
					--
					\filldraw[color=black, fill=blue!80!black] 
					(7.626153136706153,1) circle (4.5pt)
					node[anchor=south]{$P_9$};
					\filldraw[color=black, fill=blue!80!black] 
					(-2.0932994010831303,-0.027543570516784577) circle (3.5pt)
					node[anchor=north east]{$P_7$};
					\filldraw[color=black, fill=blue!80!black] 
					(2.9619226618813506,0.9915976690905064) circle (3.5pt)
					node[anchor=south]{$P_8$};
					\draw[color=black] (1,-3.2) node {(ii)};
				\end{tikzpicture}\vspace*{-0.5cm}
	\end{center}
	\caption{Two triples of concurrent lines have a common line
	\label{Figure16}}
\end{figure}
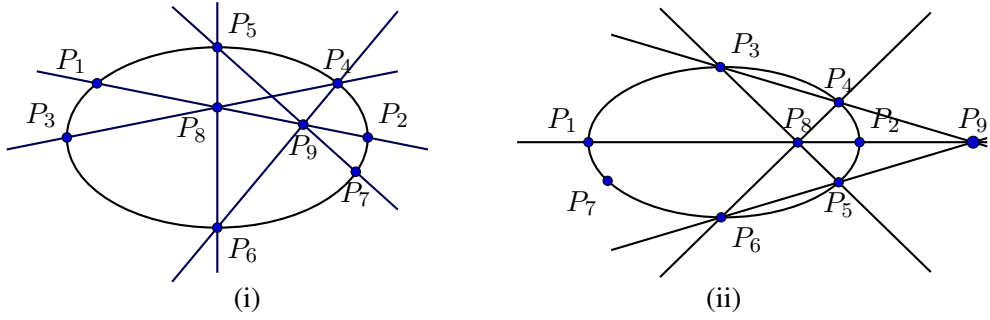
\begin{itemize}
	\item For the configuration (i), the B\'ezout decomposition 
	of $D$ with respect to $C$, $L_{12}, L_{34}, L_{56}$, 
	$L_{46}$ and $L_{57}$ is given by
	\[
	D = kC +pL_{12} + q(L_{46}+L_{56})+r(L_{34}+ L_{57}) + B(D).
	\]
	An application of Proposition~\ref{CountingMultiplicities} yields
	$$\qquad\quad  
	\begin{aligned}
		B(D)&\ge 0 &\Rightarrow&\quad d\ge 2k +p +2q +2r,\\
		B(D)\cdot C&\ge 2(m-k-p) +2(m-k-r) && \\
		&\quad +2(m-k-q-r) +(m-k-2q) 
		&\Leftrightarrow&\quad 2d\ge 7m-3k,\\
		B(D)\cdot L_{12}&\ge 2(m-k-p) +2(m-p-q-r)
		&\Leftrightarrow&\quad d\ge 4m-3p,\\
		B(D)\cdot L_{46}& \ge (m-k-q-r)+(m-k-2q)+(m-p-q-r)
		&\Leftrightarrow&\quad d\ge 3m-2q,\\
		B(D)\cdot L_{34}& \ge  (m-k-r) +(m-k-q-r) +(m-p-q-r)
		&\Leftrightarrow&\quad d\ge 3m-r.
	\end{aligned}
	$$
	It follows that $17d\ge 45m$ or $d\ge \frac{45}{17}m$.
	Consequently, we find 
	$\hat{\alpha}_{\mathbb{X}}\ge \frac{45}{17}>\frac{18}{7}$.
	
	\item For the configuration (ii), the B\'ezout decomposition 
	of $D$ with respect to $C$, $L_{12}, L_{34}, L_{56}$, $L_{35}$ 
	and $L_{46}$ is given by
	\[
	D = kC + pL_{12} + q(L_{34}+ L_{56}) + r (L_{35}+L_{46}) + B(D).
	\]
	By Proposition~\ref{CountingMultiplicities}, we find
	$$
	\begin{aligned}
		B(D)&\ge 0 &\Rightarrow&\quad d\ge 2k +p +2q +2r,\\
		B(D)\cdot C&\ge 2(m-k-p) +4(m-k-q-r) +(m-k)
		&\Leftrightarrow&\quad 2d\ge 7m-3k,\\
		B(D)\cdot L_{12}&\ge 2(m-k-p) +(m-p-2r) +(m-p-2q)
		&\Leftrightarrow&\quad d\ge 4m-3p,\\
		B(D)\cdot L_{34}& \ge 2(m-k-q-r) +(m-p-2q)
		&\Leftrightarrow&\quad d\ge 3m-2q,\\
		B(D)\cdot L_{35}& \ge 2(m-k-q-r) +(m-p-2r)
		&\Leftrightarrow&\quad d\ge 3m-2r.
	\end{aligned}
	$$
	Then we have $14d\ge 36m$ or $d\ge \frac{18}{7}m$, and hence  
	$\hat{\alpha}_{\mathbb{X}}\ge \frac{18}{7}$.
\end{itemize}

\smallskip\noindent 
\underline{\textbf{Subcase 2:}}\  
\textit{Two triples of concurrent lines do not have a common line.} 
Without loss of generality, we may assume that $\mathbb{X}=\{P_1,\dots,P_9\}$
has the configuration as in Figure~\ref{Figure17}.
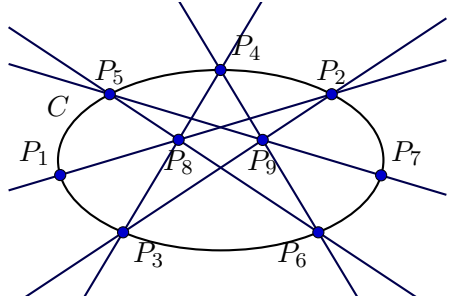
\begin{figure}[ht] 
	\begin{center}
		\begin{tikzpicture}[scale=0.6]
			\clip(-3.7,-2) rectangle (6,4.5);
			--
			\draw [rotate around={0:(1,1)},style=thick] (1,1) ellipse (3.6055512754639887cm and 2cm);
			\draw [style=thick,blue!30!black,domain=-3.7:6] plot(\x,{(--10.087081917223344-3.5999981971912587*\x)/2.168171891497525});
			\draw [style=thick,blue!30!black,domain=-3.7:6] plot(\x,{(-2.892894874303784-3.5999981971912587*\x)/-2.15848963905926});
			\draw [style=thick,blue!30!black,domain=-3.7:6] plot(\x,{(--0.7895137205823595--3.0626895260754132*\x)/4.622345382661781});
			\draw [style=thick,blue!30!black,domain=-3.7:6] plot(\x,{(--6.914892772733187-3.0626895260754132*\x)/4.622345382661781});
			\draw [style=thick,blue!30!black,domain=-3.7:6] plot(\x,{(--6.865325378347093-1.013609459563407*\x)/3.3882453747491437});
			\draw [style=thick,blue!30!black,domain=-3.7:6] plot(\x,{(-4.856859156232961-1.0109961396182705*\x)/-3.3921895118479175});
			--				
			\filldraw[color=black, fill=blue!80!black] 
			(-2.556120000651211,0.6699601231699771) circle (3.5pt)
			node[anchor=south east]{$P_1$};
			\filldraw[color=black, fill=blue!80!black] 
			(3.459014617383388,2.462689526075413) circle (3.5pt)
			node[anchor=south ]{$P_2$};
			\filldraw[color=black, fill=blue!80!black] 
			(-1.1633307652783924,-0.6) circle (3.5pt)
			node[anchor=north west]{$P_3$};
			\filldraw[color=black, fill=blue!80!black] 
			(0.9951588737808678,2.9999981971912586) circle (3.5pt)
			node[anchor=south west]{$P_4$};
			\filldraw[color=black, fill=blue!80!black] 
			(-1.4590146173833876,2.462689526075413) circle (3.5pt)
			node[anchor=south ]{$P_5$};
			\filldraw[color=black, fill=blue!80!black] 
			(3.163330765278393,-0.6) circle (3.5pt)
			node[anchor=north east]{$P_6$};
			\filldraw[color=black, fill=blue!80!black] 
			(4.554242697133779,0.6637954606670855) circle (3.5pt)
			node[anchor=south west]{$P_7$};
			\filldraw[color=black, fill=blue!80!black] 
			(0.06682510553547037,1.4516933864571426) circle (3.5pt)
			node[anchor=north]{$P_8$};
			\filldraw[color=black, fill=blue!80!black] 
			(1.929230757365756,1.449080066512006) circle (3.5pt)
			node[anchor=north]{$P_9$};
			\draw[color=black] (-2.6,2.2) node {$C$};
		\end{tikzpicture}	\vspace*{-0.5cm}
	\end{center}
	\caption{Two triples of concurrent lines do not have a common line \label{Figure17}}
\end{figure}
For a divisor $D$ of degree $d$ passing through $m\mathbb{X}$, 
the B\'ezout decomposition of $D$ with respect to $C$, 
$L_{12}$, $L_{57}$, $L_{23}$, $L_{56}$, $L_{34}$ 
and $L_{46}$ is given by
\[
D = kC + p(L_{12}+L_{57}) + q(L_{23}+ L_{56}) + r(L_{34}+L_{46}) + B(D).
\]
Since $B(D)\ge 0$, we have $d-(2k +2p +2q +2r)\ge 0$.
An application of Proposition~\ref{CountingMultiplicities} yields 
$$
\begin{aligned}
	B(D)\cdot C&\ge 2(m-k-p) +2(m-k-p-q) +2(m-k-q-r) +(m-k-2r),\\
	B(D)\cdot L_{12}&\ge (m-k-p) +(m-k-p-q) +(m-p-q-r),\\
	B(D)\cdot L_{23}& \ge (m-k-p-q) +(m-k-q-r) +(m-p-q-r),\\
	B(D)\cdot L_{34}& \ge (m-k-q-r) +(m-k-2r) + (m-p-q-r)
\end{aligned}
$$
and consequently
\begin{align}
	d&\ge 2k +2p +2q +2r, \label{Equ-421}\\
	2d&\ge 7m-3k, \label{Equ-422}\\
	d &\ge 3m-p+r, \label{Equ-423}\\ 
	d &\ge 3m-q, \label{Equ-424}\\ 
	d &\ge 3m +p-2r. \label{Equ-425}
\end{align}
Multiplying the inequalities \eqref{Equ-421}, \eqref{Equ-422},
\eqref{Equ-423}, \eqref{Equ-424}, \eqref{Equ-425} 
by $3; 2; 18; 6$ and $12$, respectively,
and then summing them up gives $d\ge \frac{122}{43}m$.
Hence we get $\hat{\alpha}_{\mathbb{X}}\ge \frac{122}{43}>\frac{18}{7}$,
as desired. 
\end{proof}

Next, we treat the case (3). Depending on the position of the points 
$P_1,\dots, P_6$, we can compute the Waldschmidt as follows.

\begin{proposition}
If $\mathbb{X}$ has the configuration as in (3), 
then the Waldschmidt constant is described as follows.
\begin{enumerate} 
\item[(3a)] If $P_i\notin L$ for $i=1,\dots, 6$, then $\hat{\alpha}_{\mathbb{X}}=3$;
\item[(3b)] If there is exactly one point among $\{P_1,\dots,P_6\}$ 
lies on  $L$, then $\frac{58}{23}\le \hat{\alpha}_{\mathbb{X}}\le 3$.
\item[(3c)] If there are exactly two points among $\{P_1,\dots,P_6\}$ lie on $L$, then $\frac{53}{21}\le \hat{\alpha}_{\mathbb{X}}\le 3$.
\end{enumerate}
\end{proposition}
\begin{proof}
Since any set of nine points is contained in a cubic curve, 
we have $\hat{\alpha}_{\mathbb{X}}\le 3.$ 
Let $D$ be a divisor of degree $d$ passing through $m\mathbb{X}$
and let  
$$
D= kC+pL + B(D)
$$
be the Bézout decomposition of $D$. 
 
In the case of (3a), by Proposition~\ref{bezout decomposition}, we have
$B(D)\cdot C = 2(d-2k-p) \ge 6(m-k)$ 
and $B(D) \cdot L=d-2k-p\ge 3(m-p).$ 
These inequalities imply $d\ge 3m.$ 
Thus, the Waldschmidt constant of $\mathbb{X}$ is
 $\hat{\alpha}_{\mathbb{X}}=3.$
 
For (3b), we apply Proposition~\ref{bezout decomposition} to $D$ and we find the following inequalities
\begin{align*}
B(D)&\ge 0 &\Rightarrow&\quad d\ge 2k+p,\\
B(D)\cdot C &= 2(d-2k-p) \ge 5(m-k)+(m-k-p)&\Leftrightarrow&\quad 2d\ge 6m-2k+p,\\
B(D)\cdot L &= d-2k-p \ge (m-k-p)+3(m-p) &\Leftrightarrow&\quad d\ge 4m+k-3p.
\end{align*}
Multiplying these last inequalities by $5, 7$ and $4$, respectively,
	and then summing them up gives $d\ge \frac{58}{23}m$.
Hence, we obtain $\hat{\alpha}_{\mathbb{X}} \ge \frac{58}{23}>\frac{5}{2}$.

To prove (3c), we may assume that $P_1,P_2\in L$ and let 
$\mathcal{H} := \bigcup_{3\le i<j\le 6} L_{ij}$ with $L_{ij}=P_iP_j$.
Furthermore, let $D'$ be a divisor of degree $d$ passing through $m\mathbb{X}$.
Consider the following subcases.

\smallskip\noindent 
\underline{\textbf{Subcase 1:}}\  
\textit{At least one of the points $P_7,P_8,P_9$ does not belong to 
	$\mathcal{H}$, say $P_7$.} 
In this case the five points $P_3,\dots,P_7$ are in general linear position
and defines an irreducible conic $C'$, see Figure~\ref{Figure18}. 
\begin{figure}[ht] 
	\begin{center}
		\begin{tikzpicture}[scale=0.55]
			\draw [rotate around={0:(1,1)},style=thick] (1,1) ellipse (3.6055512754639887cm and 2cm);
			\draw [rotate around={90:(1,1)},style=thick] (1,1) ellipse (2.828427124746189cm and 2cm);
			\draw [style=thick,domain=-5.514589872833051:7.0098818212186] plot(\x,{(--6-0*\x)/6});
			\draw[color=black] (-2.532901454209406,2.519591976812507) node {$C$};
			\draw[color=black] (-0.5299351882637512,0.357298848802995) node {$C'$};
			\draw[color=black] (-4.4888626265925593,0.6987135532255495) node {$L$};
			--				
			\filldraw[color=black, fill=blue!80!black] 
			(-0.5374122295716144,2.8090680674665816) circle (3.5pt)
			node[anchor=south east]{$P_3$};
			\filldraw[color=black, fill=blue!80!black] 
			(2.5374122295716144,2.809068067466581) circle (3.5pt)
			node[anchor=south west]{$P_4$};
			--
			\filldraw[color=black, fill=blue!80!black] 
			(-2.6055512754639887,1) circle (3.5pt)
			node[anchor=south east]{$P_1$};
			\filldraw[color=black, fill=blue!80!black] 
			(4.60555127546399,1) circle (3.5pt)
			node[anchor=south west]{$P_2$};
			--
			\filldraw[color=black, fill=blue!80!black] 
			(2.5374122295716144,-0.8090680674665813) circle (3.5pt)
			node[anchor=north west]{$P_5$};
			\filldraw[color=black, fill=blue!80!black] 
			(-0.5374122295716144,-0.8090680674665812) circle (3.5pt)
			node[anchor=north east]{$P_6$};
			--
			\filldraw[color=black, fill=blue!80!black] (-1,1) circle (3.5pt)
			node[anchor=south east]{$P_7$};
			\filldraw[color=black, fill=blue!80!black] (1,1) circle (3.5pt)
			node[anchor=south east]{$P_8$};
			\filldraw[color=black, fill=blue!80!black] (2.5374122295716144,1) circle (3.5pt)
			node[anchor=south]{$P_9$};
			\draw[color=black] (1,-3) node {(1.i)};
		\end{tikzpicture}\qquad 
		\begin{tikzpicture}[scale=0.55]
			\draw [rotate around={0:(1,1)},style=thick] (1,1) ellipse (3.6055512754639887cm and 2cm);
			\draw [rotate around={90:(1,1)},style=thick] (1,1) ellipse (2.828427124746189cm and 2cm);
			\draw [style=thick,domain=-5.514589872833051:7.0098818212186] plot(\x,{(--6-0*\x)/6});
			\draw[color=black] (-2.532901454209406,2.519591976812507) node {$C$};
			\draw[color=black] (-0.5299351882637512,0.357298848802995) node {$C'$};
			\draw[color=black] (-4.4888626265925593,0.6987135532255495) node {$L$};
			--				
			\filldraw[color=black, fill=blue!80!black] 
			(-0.5374122295716144,2.8090680674665816) circle (3.5pt)
			node[anchor=south east]{$P_3$};
			\filldraw[color=black, fill=blue!80!black] 
			(2.5374122295716144,2.809068067466581) circle (3.5pt)
			node[anchor=south west]{$P_4$};
			--
			\filldraw[color=black, fill=blue!80!black] 
			(-2.6055512754639887,1) circle (3.5pt)
			node[anchor=south east]{$P_1$};
			\filldraw[color=black, fill=blue!80!black] 
			(4.60555127546399,1) circle (3.5pt)
			node[anchor=south west]{$P_2$};
			--
			\filldraw[color=black, fill=blue!80!black] 
			(2.5374122295716144,-0.8090680674665813) circle (3.5pt)
			node[anchor=north west]{$P_5$};
			\filldraw[color=black, fill=blue!80!black] 
			(-0.5374122295716144,-0.8090680674665812) circle (3.5pt)
			node[anchor=north east]{$P_6$};
			--
			\filldraw[color=black, fill=blue!80!black] (-1,1) circle (3.5pt)
			node[anchor=south east]{$P_7$};
			\filldraw[color=black, fill=blue!80!black] (3,1) circle (3.5pt)
			node[anchor=south east]{$P_8$};
			\filldraw[color=black, fill=blue!80!black] (6,1) circle (3.5pt)
			node[anchor=south west]{$P_9$};
			\draw[color=black] (1,-3) node {(1.ii)};
		\end{tikzpicture}\vspace*{-0.5cm}
	\end{center}
	\caption{At least one of $P_7,P_8,P_9$ does not belong to 
		$\mathcal{H}$ \label{Figure18}}
\end{figure}
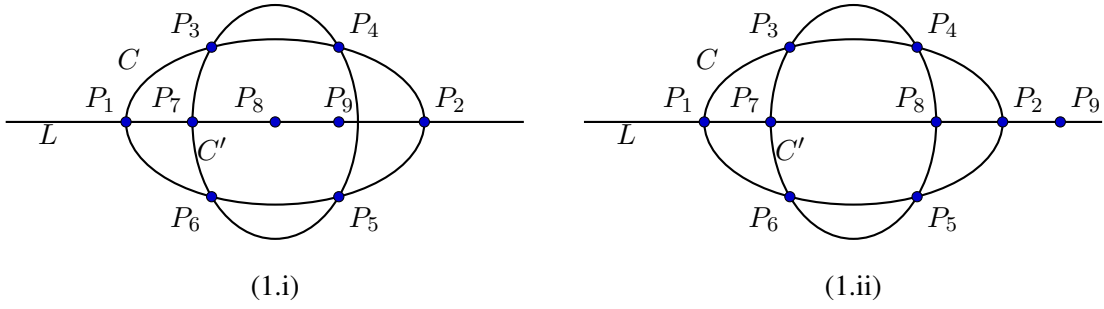

The B\'{e}zout decomposition of the divisor $D'$ with respect to 
$C, C'$ and $L$ is given by 
\begin{equation}\label{Equ-SixPointsOnConicCC}
	D'= kC +pC' +qL + B(D').	
\end{equation}

For the configuration (1.i), 
we apply Proposition~\ref{bezout decomposition} to $D'$ and receive
\begin{align*}
	B(D')&\ge 0 &\Rightarrow&\quad d\ge 2k+2p+q,\\
	B(D')\cdot C &= 2(d-2k-2p-q) \ge 4(m-k-p)+2(m-k-q)
	&\Leftrightarrow&\quad 2d\ge 6m-2k,\\
	B(D')\cdot C' &= 2(d-2k-2p-q) \ge 4(m-k-p)+(m-p-q)
	&\Leftrightarrow&\quad 2d\ge 5m-p+q,\\
	B(D')\cdot L &= d-2k-2p-q \ge 2(m-k-q) + (m-p-q) + 2(m-q) &\Leftrightarrow&\quad d\ge 5m+p-4q.
\end{align*}
Eliminating $k,p,q$ from the above inequalities, we obtain $10d \ge 26m$ or $d\ge \frac{13}{5}m$,
and subsequently $\hat{\alpha}_\mathbb{X} \ge \frac{13}{5} >\frac{5}{2}$.

For the configuration (1.ii), we also apply Proposition~\ref{bezout decomposition}
to the divisor $D'$ as in \eqref{Equ-SixPointsOnConicCC} and we have
\begin{align*}
	B(D')&\ge 0 &\Rightarrow&\quad d\ge 2k+2p+q,\\
	B(D')\cdot C &= 2(d-2k-2p-q) \ge 4(m-k-p)+2(m-k-q)&\Leftrightarrow&\quad 2d\ge 6m-2k,\\
	B(D')\cdot C' &= 2(d-2k-2p-q) \ge 4(m-k-p)+2(m-p-q)&\Leftrightarrow&\quad 2d\ge 6m-2p,\\
	B(D')\cdot L &= d-2k-2p-q \ge 2(m-k-q) + 2(m-p-q) + (m-q) &\Leftrightarrow&\quad d\ge 5m-4q.
\end{align*}
The above inequalities yields $21d \ge 53m$ or $d\ge \frac{53}{21}m$,
and subsequently $\hat{\alpha}_\mathbb{X} \ge \frac{53}{21} >\frac{5}{2}$.

\smallskip\noindent 
\underline{\textbf{Subcase 2:}}\  
\textit{All $P_7,P_8,P_9$ belong to $\mathcal{H}$ and none of them is
the intersection point of a pair of lines $L_{ij}$ with $3\le i<j\le 6$.}
Suppose that $P_7$, $P_8$, $P_9$ lie on lines 
$L_{ij}$, $L_{kl}$, $L_{st}$ in $\mathcal{H}$, respectively.
Since $P_i,P_j,P_k,P_l,P_s,P_t\in\{P_3,P_4,P_5,P_6\}$, $i<j$, $k<l$ and $s<t$, 
we may assume without loss of generality that $P_j=P_l$
and $P_e\in \{P_3,P_4,P_5,P_6\}\setminus\{P_i,P_j,P_k\}$.
Then $P_i,P_k,P_e,P_7,P_8$ are in general linear position
and define an irreducible conic $C'$. 
Thus the configuration of $\mathbb{X}$ can be described as in Figure~\ref{Figure019}.
\begin{figure}[ht] 
	\begin{center}
		\begin{tikzpicture}[scale=0.6]
			\draw [rotate around={0:(1,1)},style=thick] (1,1) ellipse (3.6055512754639887cm and 2cm);
			\draw [style=thick,domain=-5:8.2] plot(\x,{(--6-0*\x)/6});
			\draw [style=thick,dashed,domain=-0.65:0.25] plot(\x,{(-1.377408556381528-3.7488820648669248*\x)/-0.4910277663518987});
			\draw [style=thick,dashed,domain=-1.1:3.2] plot(\x,{(--7.173939784619538-3.755385986231659*\x)/2.4334656654191167});
			\draw [rotate around={-1.1:(0.78,0.059)},style=thick] (0.78,0.059) ellipse (3.861834171450144cm and 0.9566763517957737cm);
			\draw[style=thick,dashed,domain=-3.5:8] plot(\x,{(-3.4701491094293435--1.1531884015853833*\x)/4.871594954090374});
			\draw[color=black] (-2.532901454209406,2.519591976812507) node {$C$};
			\draw[color=black] (-3.2,-0.5) node {$C'$};
			\draw[color=black] (-4.5,0.6) node {$L$};
			--				
			\filldraw[color=black, fill=blue!80!black] 
			(-0.47546008982930765,-0.8248723894960128) circle (3.5pt)
			node[anchor=north east]{$P_3$};
			\filldraw[color=black, fill=blue!80!black] 
			(2.449033341941708,-0.831376310860747) circle (3.5pt)
			node[anchor=north east]{$P_4$};
			--
			\filldraw[color=black, fill=blue!80!black] 
			(-2.6055512754639887,1) circle (3.5pt)
			node[anchor=south east]{$P_1$};
			\filldraw[color=black, fill=blue!80!black] 
			(4.60555127546399,1) circle (3.5pt)
			node[anchor=south west]{$P_2$};
			--
			\filldraw[color=black, fill=blue!80!black] 
			(0.015567676522591056,2.924009675370912) circle (3.5pt)
			node[anchor=south west]{$P_6$};
			\filldraw[color=black, fill=blue!80!black] 
			(4.396134864261066,0.32831601208937056) circle (3.5pt)
			node[anchor=north west]{$P_5$};
			--
			\filldraw[color=black, fill=blue!80!black] 
			(-0.23643869684150579,1) circle (3.5pt)
			node[anchor=south east]{$P_7$};
			\filldraw[color=black, fill=blue!80!black]
			(1.2623134177366542,1) circle (3.5pt)
			node[anchor=south west]{$P_8$};
			\filldraw[color=black, fill=blue!80!black] 
			(7.233635069561602,1) circle (3.5pt)
			node[anchor=south]{$P_9$};
		\end{tikzpicture}\vspace*{-0.3cm}
	\end{center}
	\caption{$P_3,P_4,P_5,P_7,P_8$ lie on an irreducible conic $C'$ and 
		$P_6\notin C'$ \label{Figure019}}
\end{figure}
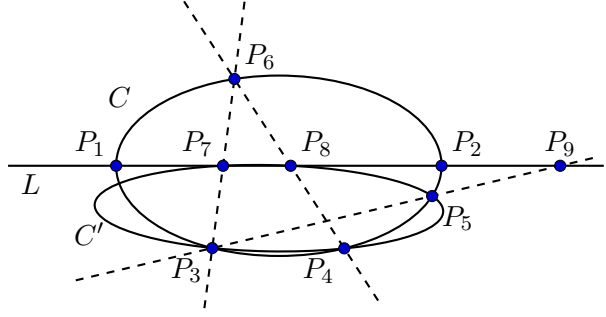

Consider a divisor $D'$ of degree $d$ passing through $m\mathbb{X}$ 
with the Bézout decomposition given by 
\[
D' = kC + pC' + qL + B(D').
\]
Then we have  
\begin{align}
B(D')&\ge 0 &\Rightarrow&\quad d\ge 2k+2p+q, \label{equ051}\\
B(D')\cdot C & \ge 2(m-k-q) +3(m-k-p) +(m-k) 
&\Leftrightarrow&\quad 2d\ge 6m-2k+p, \label{equ052}\\
B(D')\cdot C' & \ge 3(m-k-p)+2(m-p-q)
&\Leftrightarrow&\quad 2d\ge 5m+k-p, \label{equ053}\\
B(D')\cdot L & \ge 2(m-k-q) + 2(m-p-q) + (m-q) 
&\Leftrightarrow&\quad d\ge 5m-4q. \label{equ054}
\end{align}
Multiplying the inequalities \eqref{equ051}, \eqref{equ052},
\eqref{equ053}, \eqref{equ054} by $4; 16; 24; 1$, respectively,
and then summing them up provides $85d \ge 221m$ or $d\ge \frac{13}{5}m$,
and hence $\hat{\alpha}_\mathbb{X} \ge \frac{13}{5} >\frac{5}{2}$.

\smallskip\noindent 
\underline{\textbf{Subcase 3:}}\  
\textit{All $P_7,P_8,P_9$ belong to $\mathcal{H}$ and exactly one of them, 
	say $P_7$, is the intersection point of a pair of lines $L_{ij}$ 
	with $3\le i<j\le 6$.} 
In this case $\mathbb{X}$ can be assumed to have one of the following
configurations (see Figure~\ref{Figure021}).
\begin{figure}[ht] 
	\begin{center}
		\begin{tikzpicture}[scale=0.55]
			\draw [rotate around={0:(1,1)},style=thick] (1,1) ellipse (3.6055512754639887cm and 2cm);
			\draw [style=thick,domain=-4.5:6] plot(\x,{(--6-0*\x)/6});
			\draw [style=thick] (-0.5374122295716144,-2.271594375250676) -- (-0.5374122295716144,4.556699713200414);
			\draw [style=thick] (2.537412229571615,-2.271594375250676) -- (2.537412229571615,4.556699713200414);
			\draw [style=thick,domain=-1.9:3.7] plot(\x,{(--6.692960594076392-3.618136134933163*\x)/3.0748244591432288});
			\draw [style=thick,domain=-1.9:3.7] plot(\x,{(--0.5433116757899339-3.618136134933162*\x)/-3.0748244591432288});
			
			\draw[color=black] (-2.532901454209406,2.519591976812507) node {$C$};
			\draw[color=black] (-4,0.5) node {$L$};
			--				
			\filldraw[color=black, fill=blue!80!black] 
			(-0.5374122295716144,2.8090680674665816) circle (3.5pt)
			node[anchor=south west]{$P_3$};
			\filldraw[color=black, fill=blue!80!black] 
			(2.5374122295716144,2.809068067466581) circle (3.5pt)
			node[anchor=south east]{$P_4$};
			--
			\filldraw[color=black, fill=blue!80!black] 
			(-2.6055512754639887,1) circle (3.5pt)
			node[anchor=south east]{$P_1$};
			\filldraw[color=black, fill=blue!80!black] 
			(4.60555127546399,1) circle (3.5pt)
			node[anchor=south west]{$P_2$};
			--
			\filldraw[color=black, fill=blue!80!black] 
			(2.5374122295716144,-0.8090680674665813) circle (3.5pt)
			node[anchor=north east]{$P_6$};
			\filldraw[color=black, fill=blue!80!black] 
			(-0.5374122295716144,-0.8090680674665812) circle (3.5pt)
			node[anchor=north west]{$P_5$};
			--
			\filldraw[color=black, fill=blue!80!black] 
			(-0.5374122295716144,1) circle (3.5pt)
			node[anchor=south east]{$P_8$};
			\filldraw[color=black, fill=blue!80!black] (1,1) circle (3.5pt)
			node[anchor=south]{$P_7$};
			\filldraw[color=black, fill=blue!80!black] 
			(2.537412229571615,1) circle (3.5pt)
			node[anchor=south west]{$P_9$};
			\draw[color=black] (1,-3) node {(2.i)};
		\end{tikzpicture}\qquad\
		\begin{tikzpicture}[scale=0.55]
			\draw [rotate around={0:(1,1)},style=thick] (1,1) ellipse (3.6055512754639887cm and 2cm);
			\draw [style=thick,domain=-4.5:8] plot(\x,{(--6-0*\x)/6});
			\draw [style=thick,domain=-2:8] plot(\x,{(-3.9846050478594517--0.931986882737109*\x)/3.122869641295142});
			\draw [style=thick,domain=-2:8] plot(\x,{(--21.97118556316587-1.9992549245813636*\x)/6.724561348994481});
			\draw [style=thick] (4.04912177702199,-2) -- (4.04912177702199,4.25);
			\draw [style=thick] (0.9260321852423081,-2) -- (0.9260321852423081,4.25);
			--
			\draw[color=black] (-2.532901454209406,2.519591976812507) node {$C$};
			\draw[color=black] (-4,0.5) node {$L$};
			--				
			\filldraw[color=black, fill=blue!80!black] 
			(0.9015917877116715,2.9992549245813636) circle (3.5pt)
			node[anchor=south east]{$P_3$};
			\filldraw[color=black, fill=blue!80!black] 
			(4.055829998181412,2.061479812078704) circle (3.5pt)
			node[anchor=south west]{$P_4$};
			--
			\filldraw[color=black, fill=blue!80!black] 
			(-2.6055512754639887,1) circle (3.5pt)
			node[anchor=south east]{$P_1$};
			\filldraw[color=black, fill=blue!80!black] 
			(4.60555127546399,1) circle (3.5pt)
			node[anchor=south west]{$P_2$};
			--
			\filldraw[color=black, fill=blue!80!black] 
			(4.04890182653745,-0.0675922085395424) circle (3.5pt)
			node[anchor=north west]{$P_6$};
			\filldraw[color=black, fill=blue!80!black] 
			(0.9260321852423081,-0.9995790912766513) circle (3.5pt)
			node[anchor=north east]{$P_5$};
			--
			\filldraw[color=black, fill=blue!80!black] 
			(7.626153136706153,1) circle (3.5pt)
			node[anchor=south]{$P_7$};
			\filldraw[color=black, fill=blue!80!black] 
			(0.9138109958428614,1) circle (3.5pt)
			node[anchor=south east]{$P_8$};
			\filldraw[color=black, fill=blue!80!black] 
			(4.052375857474351,1) circle (3.5pt)
			node[anchor=south east]{$P_9$};
			\draw[color=black] (1,-3) node {(2.ii)};
		\end{tikzpicture}\vspace*{-0.5cm}
	\end{center}
	\caption{All $P_7,P_8,P_9$ belong to $\mathcal{H}$ and only $P_7$ is the intersection point of a pair of lines $L_{ij}$ with $3\le i<j\le 6$ \label{Figure020}}
\end{figure}
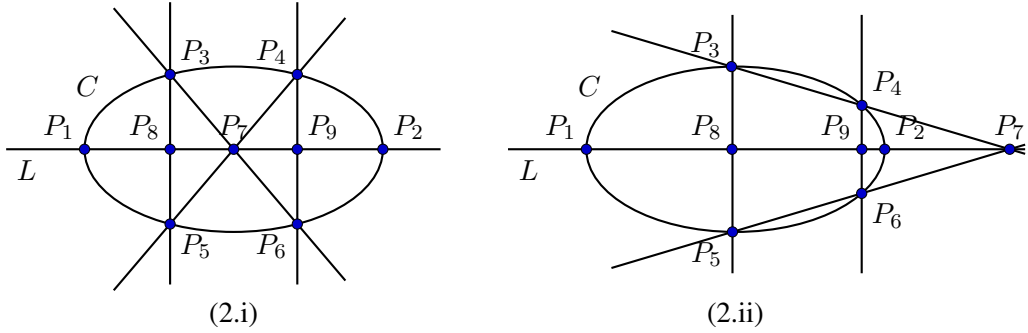

We shall show that $\hat{\alpha}_\mathbb{X} = \frac{13}{5} >\frac{5}{2}$
for the configuration (2.i) (similar for (2.ii)).
In this case, the B\'{e}zout decomposition of the divisor $D'$ with respect to 
$C$, $L$, $L_{35}, L_{46}$, $L_{36}$ and $L_{45}$ is given by 
\[
D'= kC +pL +q(L_{35}+L_{46})+r(L_{36}+L_{45}) + B(D').	
\]
By Proposition~\ref{bezout decomposition}, we have
\begin{align*}
	B(D')&\ge 0 &\Rightarrow&\quad d\ge 2k+p+2q+2r,\\
	B(D')\cdot C & \ge 2(m-k-p)+ 4(m-k-q-r)
	&\Leftrightarrow&\quad 2d\ge 6m-2k,\\
	B(D')\cdot L & \ge 2(m-k-p) +2(m-p-q) +(m-p-2r)
	&\Leftrightarrow&\quad d\ge 5m -4p,\\
	B(D')\cdot L_{35} &\ge 2(m-k-q-r)+(m-p-q), 
	&\Leftrightarrow&\quad d\ge 3m-q\\ 
	B(D')\cdot L_{36} &\ge 2(m-k-q-r)+(m-p-2r) ,
	&\Leftrightarrow&\quad d\ge 3m -2r.
\end{align*}
These inequalities yields $25d\ge 65m$ or $d\ge \frac{13}{5}m$,
and consequently $\hat{\alpha}_\mathbb{X} \ge \frac{13}{5}$.
Moreover, the divisor 
$
D'' = m(L_{36}+L_{45}+2(L_{35}+L_{46})+3L_{12}+2C)
$
is of degree $13m$ and vanishes along $\mathbb{X}$ of multiplicity $5m$,
and this implies $\hat{\alpha}_\mathbb{X} \le \frac{13}{5}$.
Hence $\hat{\alpha}_\mathbb{X} =\frac{13}{5}$, as wanted.

\smallskip\noindent 
\underline{\textbf{Subcase 4:}}\  
\textit{All $P_7,P_8,P_9$ belong to $\mathcal{H}$ and two of them, 
	say $P_7,P_8$, are the intersection points of pairs of lines 
	$L_{ij}$ with $3\le i<j\le 6$.} 
In this case $\mathbb{X}$ has the following configuration 
(see Figure~\ref{Figure021}). 
\begin{figure}[ht] 
	\begin{center}
	\begin{tikzpicture}[scale=0.55]
		\draw [rotate around={0:(1,1)},style=thick] (1,1) ellipse (3.6055512754639887cm and 2cm);
		\draw [style=thick,domain=-4.5:8] plot(\x,{(--6-0*\x)/6});
		\draw [style=thick,domain=-2:8] plot(\x,{(-3.9846050478594517--0.931986882737109*\x)/3.122869641295142});
		\draw [style=thick,domain=-2:8] plot(\x,{(--21.97118556316587-1.9992549245813636*\x)/6.724561348994481});
		\draw [style=thick,domain=-0.7:6.5] plot(\x,{(--12.204629322521468-3.0668471331209055*\x)/3.147310038825793});
		\draw [style=thick,domain=-0.7:6.5] plot(\x,{(--5.963119519166904-3.061058903355351*\x)/-3.129797812939109});
		\draw [style=thick] (0.9260321852423081,-2) -- (0.9260321852423081,4.25);
		--
		\draw[color=black] (-2.532901454209406,2.519591976812507) node {$C$};
		\draw[color=black] (-4,0.5) node {$L$};
		--				
		\filldraw[color=black, fill=blue!80!black] 
		(0.9015917877116715,2.9992549245813636) circle (3.5pt)
		node[anchor=south west]{$P_3$};
		\filldraw[color=black, fill=blue!80!black] 
		(4.055829998181412,2.061479812078704) circle (3.5pt)
		node[anchor=south]{$P_4$};
		--
		\filldraw[color=black, fill=blue!80!black] 
		(-2.6055512754639887,1) circle (3.5pt)
		node[anchor=south east]{$P_1$};
		\filldraw[color=black, fill=blue!80!black] 
		(4.60555127546399,1) circle (3.5pt)
		node[anchor=south west]{$P_2$};
		--
		\filldraw[color=black, fill=blue!80!black] 
		(4.04890182653745,-0.0675922085395424) circle (3.5pt)
		node[anchor=north]{$P_6$};
		\filldraw[color=black, fill=blue!80!black] 
		(0.9260321852423081,-0.9995790912766513) circle (3.5pt)
		node[anchor=north west]{$P_5$};
		--
		\filldraw[color=black, fill=blue!80!black] 
		(7.626153136706153,1) circle (3.5pt)
		node[anchor=south]{$P_7$};
		\filldraw[color=black, fill=blue!80!black] 
		(0.9138109958428614,1) circle (3.5pt)
		node[anchor=south east]{$P_9$};
		\filldraw[color=black, fill=blue!80!black] 
		(2.9619226618813506,0.9915976690905064) circle (3.5pt)
		node[anchor=south]{$P_8$};
	\end{tikzpicture}
	\vspace*{-0.5cm}
	\end{center}
	\caption{All $P_7,P_8,P_9$ belong to $\mathcal{H}$ and
	$P_7,P_8$ are the intersection points of pairs of lines 
	$L_{ij}$ with $3\le i<j\le 6$.	 \label{Figure021}}
\end{figure}
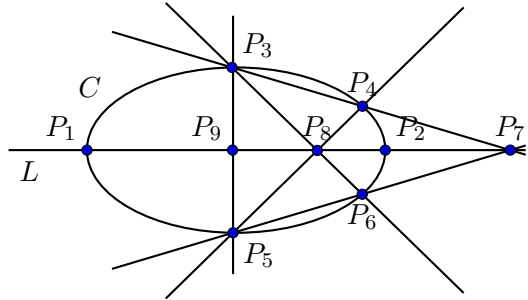
Consider the divisor $D'$ of degree $d$ passing through $m\mathbb{X}$.
Then the B\'{e}zout decomposition of $D'$ with respect to 
$C$, $L$, $L_{36}, L_{45}$, $L_{34}$, $L_{56}$ and $L_{35}$ is given by 
\[
D'= kC +pL +q(L_{36}+L_{45})+r(L_{34}+L_{56}) + sL_{35}+ B(D').	
\]
By Proposition~\ref{bezout decomposition}, we have
\begin{align*}
	&B(D')\ge 0 \hspace*{8.7cm}\Rightarrow d\ge 2k\!\!\!
	&&\!\!\!\!\!\!\!\!\!+p+2q+2r+s,\\
	&B(D')\cdot C  \ge 2(m-k-p)+ 2(m-k-q-r-s) +2(m-k-q-r)
		&\Leftrightarrow&\quad 2d\ge 6m-2k,\\
	&B(D')\cdot L \ge 2(m-k-p) +(m-p-2r)+(m-p-2q) +(m-p-s)
		&\Leftrightarrow&\quad d\ge 5m -4p,\\
	&B(D')\cdot L_{34} \ge (m-k-q-r-s)+(m-k-q-r)+(m-p-2r) 
	&\Leftrightarrow&\quad d\ge 3m -2r,\\
	&B(D')\cdot L_{35} \ge 2(m-k-q-r-s)+(m-p-s) 
		&\Leftrightarrow&\quad d\ge 3m-2s,\\ 
	&B(D')\cdot L_{36} \ge (m-k-q-r-s)+(m-k-q-r)+(m-p-2q) 
		&\Leftrightarrow&\quad d\ge 3m -2q.
\end{align*}
These inequalities yields $23d\ge 59m$ or $d\ge \frac{59}{23}m$,
and therefore $\hat{\alpha}_\mathbb{X} \ge \frac{59}{23} >\frac{5}{2}$.
\end{proof}

Finally, we treat the case (4) that five points $P_1,\dots,P_5$ lie on $C$ 
and four others  $P_6,\dots,P_9$ lie on a line~$L$. For this case, 
the line $L$ contains none of the points $P_1,\dots,P_5$, because otherwise
we would have that $\mathbb{X}=\{P_1,\dots,P_9\}$ is contained in
the union of three lines.

\begin{proposition}
If $\mathbb{X}$ has the configuration as in (4), then 
$\hat{\alpha}_{\mathbb{X}}\ge \frac{14}{5} >\frac{5}{2}.$
\end{proposition}
\begin{proof} 
Let $D$ be a divisor of degree $d$ passing through $m\mathbb{X}$
and let  
$$
D= kC+pL + B(D)
$$
be the B\'{e}zout decomposition of $D$. We have 
	$$
	\begin{aligned}
		B(D)\cdot C&\ge 5(m-k) \quad &\Leftrightarrow \quad & 2d\ge 5m-k+2p \label{eq41}, \\
		B(D)\cdot L& \ge 4(m-p) \quad &\Leftrightarrow \quad & d\ge 4m+2k-3p \label{eq42}.
	\end{aligned}
	$$
Multiplying the first inequation with 2 and then adding it to the second
yields $5d\ge 14m$, or equivalently $d\ge\frac{14}{5}m$.  
Hence we get $\hat{\alpha}_{\mathbb{X}}\ge \frac{14}{5}.$
\end{proof}

As a direct consequence of the above propositions, 
we have the following corollary.

\begin{corollary}
Let $\mathbb{X}=\{P_1,\dots,P_9\}$ be a set of nine points in $\mathbb{P}^2$
that are not contained in the union of three lines.
Then $\hat{\alpha}_{\mathbb{X}}=\frac{5}{2}$ if and only if
there are eight points of $\mathbb{X}$ lying on an irreducible conic and
the last point is exactly the intersection 
of four concurrent lines,
each passing through a pair of the eight points.
\end{corollary}

%
%

\begin{thebibliography}{00}
	
	\bibitem {Bocci2016}
	C. Bocci, S. Cooper, E. Guardo, B. Harbourne, U. Nagel, A. Seceleanu,
	A. Van Tuyl, and T. Vu, 
	The Waldschmidt constant for squarefree monomial ideals,
	J. Algebraic Combin. \textbf{44} (2016), 875--904.
	
	\bibitem {Bocci-Harbourne2010}
	C. Bocci and  B. Harbourne, 
	Comparing powers and symbolic powers of ideals,
	J. Algebraic Geom. \textbf{19} (2010), 399--417.
	
	\bibitem {Enrico2020}
	E. Carlini, H. T. H\`a, B. Harbourne, and  A. Van Tuyl, 
	\textit{Ideals of powers and powers of ideals},
	Vol. 27, Lecture Notes of the Unione Matematica Italiana,
	Springer, Cham, 2020.
	
	\bibitem {Catalisano-Favacchio-Guardo-Shin2024}
	M. V. Catalisano, G. Favacchio, E. Guardo, Elena, and Y.-S. Shin,  
	The Waldschmidt constant of a standard $\Bbbk$-configuration 
	in~$\mathbb{P}^2$, 
	Rev. Mat. Complut. (2024), available at 
	\url{https://doi.org/10.1007/s13163-024-00493-6}.
	
	\bibitem {Virginia-Elena-Su2020}
	M. V. Catalisano, E. Guardo, Elena, and Y.-S. Shin,  
	The {W}aldschmidt constant of special {$\Bbbk$}-configurations
	in {$\Bbb P^n$}, 
	J. Pure Appl. Algebra \textbf{224} (2020), 106341. 

	\bibitem {Chudnovsky1979}
	G. V. Chudnovsky,  
	Singular points on complex hypersurfaces and multidimensional Schwarz lemma, Seminar on Number Theory, Paris 1979--80, 29--69, Progr. Math 12.

	\bibitem {Demailly1982}
	J.-P. Demailly,  Formules de Jensen en plusieurs variables et applications arithm{\'e}tiques, 
	Bull. Soc. Math. France \textbf{110} (1982), 75--102.
	
	\bibitem {Michael2019}
	M. DiPasquale, C. A. Francisco, J. Mermin, and J. Schweig,  
	Asymptotic resurgence via integral closures,
	Trans. Amer. Math. Soc. \textbf{372} (2019), 6655--6676.

	\bibitem {Dumnicki2015}
 	M. Dumnicki, B. Harbourne, U. Nagel, A. Seceleanu, T. Szemberg,
 	and H. Tutaj-Gasi\'nska, 
 	Resurgences for ideals of special point configurations in
	$\mathbb{P}^N$ coming from hyperplane arrangements,
 	J. Algebra \textbf{443} (2015), 383--394.

	\bibitem {Dumnicki2014}
	M. Dumnicki, B. Harbourne, T. Szemberg, and H. Tutaj-Gasi\'nska, 
	Linear subspaces, symbolic powers and {N}agata type conjectures,
	Adv. Math. \textbf{252} (2014), 471--491.
	
	\bibitem {Dumnicki-Szemberg-Tutaj2016}
	M. Dumnicki, T. Szemberg, and H. Tutaj-Gasi\'nska, 
	Symbolic powers of planar point configurations {II},
	J. Pure Appl. Algebra \textbf{220} (2016), 2001-2016.
	
	\bibitem {FGHLMS2017}
	{\L}. Farnik, J. Gwo{\'z}dziewicz, B. Hejmej, M. Lampa-Baczy{\'n}ska,
	G. Malara, and J. Szpond, 
	Initial sequences and {W}aldschmidt constants of planar point configurations,
	Internat. J. Algebra Comput. \textbf{27} (2017), 717--729.
	
	\bibitem {Malara2018}
	G. Malara, T. Szemberg, and J. Szpond, 
	On a conjecture of {D}emailly and new bounds on {W}aldschmidt 
	constants in $\mathbb{P}^N$,
	J. Number Theory \textbf{189} (2018), 211--219.

	\bibitem {Moreau1980}
 	J.-C. Moreau,  
 	Lemmes de Schwarz en plusieurs variables et applications arithm{\'e}tiques,
 	In S{\'e}minaire Pierre Lelong-Henri Skoda (Analyse) Ann{\'e}es 1978/79,
 	pages 174--190, Springer, 1980.
 	
 	\bibitem {Mosakhani-Haghighi2016}
	M. Mosakhani and H. Haghighi, 
	On the configurations of points in $\mathbb{P}^2$ with the 
	{W}aldschmidt constant equal to two,
	J. Pure Appl. Algebra \textbf{220} (2016), 3821--3825.

	\bibitem {Waldschmidt1977}
	M. Waldschmidt,  
	Propri{\'e}t{\'e}s arithm{\'e}tiques de fonctions de plusieurs variables (II),
 	In S{\'e}minaire Pierre Lelong (Analyse) Ann{\'e}e 1975/76: Journ{\'e}es sur les Fonctions Analytiques, Toulouse, 1976, pages 108--135,
 	Springer 1977.
 	
\end{thebibliography}

\goodbreak

\address
\end{document}